\documentclass[%
 aip,
 jmp,%
 amsmath,amssymb,
preprint,%
]{revtex4-1}

\usepackage{graphicx}% Include figure files
\usepackage{dcolumn}% Align table columns on decimal point
\usepackage{bm}% bold math
\usepackage[T1]{fontenc}
\usepackage{lmodern}
\usepackage{amssymb}

\usepackage{longtable}
\usepackage{fancyvrb}
\usepackage[colorlinks=true,citecolor=blue,hypertexnames=false]{hyperref}
\usepackage{color, comment}
\usepackage{amsthm}

\newtheorem{defi}{Definition}
\newtheorem{rem}{Remark}
\newtheorem{thm}{Theorem}

\newtheorem{lem}{Lemma}
\newtheorem{cor}{Corollary}

\newcommand{\GAMMA}{\Gamma}

\def\Sp{\operatorname{Sp}}
\def\arctanh{\operatorname{arctanh}}
\def\C{\mathbb{C}}
\def\res{\operatorname{res}}
\def\Res{\operatorname{Res}}
\def\dilog{\operatorname{dilog}}

\begin{document}

%\preprint{AIP/123-QED}

\title[Integrability conditions for homogeneous potentials]{Third order integrability conditions for homogeneous potentials of degree -1}
\author{Thierry Combot}
\affiliation{IMCCE (France)}
\email{combot@imcce.fr}
\author{Christoph Koutschan}
\affiliation{RISC, Johannes Kepler University Linz (Austria)}
\email{ckoutsch@risc.jku.at}
\thanks{$^1$ Supported by University Paris 7 Diderot.}
\thanks{$^2$ Partially supported by the DDMF project of the MSR-INRIA Joint Centre, 
and by the Austrian Science Fund (FWF): P20162-N18.}

% \title[Sample title]{Sample Title:\\with Forced Linebreak\footnote{Error!}}% Force line breaks with \\
% \thanks{Footnote to title of article.}
% 
% \author{A. Author}
%  \altaffiliation[Also at ]{Physics Department, XYZ University.}%Lines break automatically or can be forced with \\
% \author{B. Author}%
%  \email{Second.Author@institution.edu.}
% \affiliation{ 
% Authors' institution and/or address%\\This line break forced with \textbackslash\textbackslash
% }%
% 
% \author{C. Author}
%  \homepage{http://www.Second.institution.edu/~Charlie.Author.}
% \affiliation{%
% Second institution and/or address%\\This line break forced% with \\
% }%

\date{\today}% It is always \today, today,
             %  but any date may be explicitly specified

\begin{abstract}
We prove an integrability criterion of order~$3$ for a homogeneous
  potential of degree~$-1$ in the plane. Still, this criterion depends
  on some integer and it is impossible to apply it directly except for
  families of potentials whose eigenvalues are bounded. To address
  this issue, we use holonomic and asymptotic computations with error
  control of this criterion and apply it to the potential of the form
  $V(r,\theta)=r^{-1} h(\exp(i\theta))$ with $h\in\mathbb{C}[z],\;\deg
  h \leq 3$. We then find all meromorphically integrable potentials of
  this form.
\end{abstract}

\pacs{02.30.Hq 02.70.Wz}% PACS, the Physics and Astronomy
                             % Classification Scheme.
\keywords{Non-integrability, Homogeneous potentials, Differential Galois theory, Holonomicity}%Use showkeys class option if keyword
                              %display desired
\maketitle

\section{Introduction}

In this article, we will be interested in non-integrability proofs of 
meromorphic homogeneous potentials of degree $-1$ in the plane, and in 
particular in nongeneric cases. Writing our potential~$V$ in polar 
coordinates, and making the Fourier expansion in the angle gives us
\begin{equation}\label{init}
V(r,\theta)=r^{-1}\sum\limits_{k=-\infty}^\infty a_k e^{ik\theta}.
\end{equation}
This type of potential covers many physical problems in celestial
mechanics and $n$-body problems, in particular the anisotropic Kepler
problem, the isosceles $3$-body problem, the colinear $3$-body
problem, the symmetric $4$-body problem and so on. Moreover, for such
a potential there are strong integrability conditions, thanks to the
Morales-Ramis theory~\cite{11} and to a very effective criterion of
Yoshida~\cite{1}. Still, for such a general potential, this criterion
will not be sufficient. This is not particularly because this class of
potentials is large, but because there are nongeneric, very resistant
cases inside.  For example, if we want to study the integrability of
$V(r,\theta)=r^{-1} h(\exp(i\theta))$ with a polynomial~$h$, we have a
priori a potential with $\deg h +1$ complex parameters, and Yoshida's
integrability criterion will restrict this family to a family with
$\deg h-1$ integer parameters. Still one would like to have a finite list of
possible integrable potentials, so as to be able to check
the existence of first integrals one by one. Here we will present a
stronger criterion in Theorems~\ref{thm:main1} and \ref{thm:main3}
which is able to deal with such families, and which therefore is capable to settle
any integrability question on finite dimensional families of
type~\eqref{init}. As an application of our method, we will apply this
criterion in the case $V(r,\theta)=r^{-1} h(\exp(i\theta))$ with
$h\in\mathbb{C}[z],\;\deg h \leq 3$. To do precise statements, let us
now begin with some definitions concerning homogeneous potentials and
integrability.

\begin{defi} We consider the algebraic variety $\mathcal{S}=\{(q_1,q_2,r)\in\mathbb{C}^3,\;\;r^2=q_1^2+q_2^2\}$ and the derivations for a function $f$ on $\mathcal{S}$
\[ \frac{\partial f}{\partial q_1}= \partial_1 f +r^{-1}q_1 \partial_3 f,\quad 
\frac{\partial f}{\partial q_2}= \partial_2 f +r^{-1}q_2 \partial_3 f \]
where $\partial_i$ is the derivative according to the $i$-th variable (the variables of $f$ are $q_1,q_2,r$ in this order). This defines a symplectic form on $\mathbb{C}^2\times \mathcal{S}$ on which we consider a Hamiltonian $H$ of the form
\[
  H(p_1,p_2,q_1,q_2,r)=\frac{1}{2}(p_1^2+p_2^2)-V(q_1,q_2,r)
\]
with the associated system of differential equations
\begin{equation}\label{eq2}
  \dot{r}=r^{-1}(q_1\dot{q}_1+q_2\dot{q}_2),\quad \dot{q}_i=\frac{\partial}{\partial p_i} H, \quad 
  \dot{p}_i=-\frac{\partial}{\partial q_i} H, \quad i=1,2.
\end{equation}
The potential $V$ is assumed to be meromorphic on $\mathcal{S}$ and to have the following form in polar coordinates:
\[
  V(r,\theta) = \frac{1}{r} U(\theta), \quad r\cos\theta=q_1,\;\; r\sin\theta = q_2
\]
This implies that $V$ is homogeneous of degree~$-1$. We say that $I$ is a meromorphic first integral of $H$, if $I$ is a meromorphic function on $\mathbb{C}^2\times \mathcal{S}$ such that
\[
  \dot{I}=\{H,I\}=\sum\limits_{i=1}^2 \left(\frac{\partial}{\partial p_{i}} H \frac{\partial}{\partial q_{i}} I-
  \frac{\partial}{\partial q_{i}} H \frac{\partial}{\partial p_{i}} I\right)=0.
\]
Obviously, the Hamiltonian~$H$ itself is a first integral. We will say that $V$ is meromorphically integrable if it possesses an additional meromorphic first integral which is independent almost everywhere from~$H$.
\end{defi}

\begin{defi} We call $c=(c_1,c_2,c_3)\in\mathcal{S}$ a Darboux point of $V$ if
\begin{equation}\label{eq.Darboux}
  \frac{\partial}{\partial q_1} V(c)= \alpha c_1\quad\hbox{and}\quad 
  \frac{\partial}{\partial q_2} V(c)= \alpha c_2
\end{equation}
where $\alpha\in\C$ is called the multiplicator. Because $V$ has
singularity at $c_3=0$, we will \textbf{always} assume that $c_3\neq0$. Because of homogeneity, we can always
choose $\alpha=0$ or $\alpha=-1$. We say that $c$ is non-degenerate
if $\alpha\neq 0$. To the Darboux point~$c$ we associate a homothetic
orbit given by
\begin{equation}\label{eq.orbit}
  r(t)=c_3 \phi(t),\;\; q_i(t)=c_i\phi(t),\quad p_i(t)=c_i \dot{\phi}(t)\qquad (i=1,2),
\end{equation}
with $\phi$ satisfying the following differential equation
$$\frac{1} {2} \dot{\phi}(t)^2=-\frac{\alpha} {\phi(t)}+E, \quad E\in\mathbb{C}.$$
\end{defi}

In the following, we will often omit the last component of a Darboux point $c\in\mathcal{S}$ as it is defined up to a sign (and the choice of sign does not matter) by the two first components.

\begin{defi} The first order variational equation of $H$ near a homothetic orbit is given by
$$\ddot{X}(t)=\frac{1}{\phi(t)^3} \nabla^2 V(c) X(t)$$
where $\nabla^2 V(c)$ is the Hessian of $V$ (according to derivations in $q$). After diagonalization (if possible) and the change of variable $\phi(t)\longrightarrow t$, 
the equation simplifies to
\[
  2t^2(Et+1)\ddot{X}_i-t\dot{X}_i=\lambda_i X_i, 
\]
where the  $\lambda_i$ are the eigenvalues of the Hessian of~$V$ evaluated at the Darboux point~$c$,
i.e., $\lambda_i\in \Sp\left(\nabla^2 V(c)\right)$.
\end{defi}

\begin{thm} \label{thm:Morales}(Morales, Ramis, Yoshida
\cite{1}\cite{2}\cite{11}\cite{12}) If $V$ is meromorphically integrable, then the neutral component of the Galois group of the variational equation near a homothetic orbit with $E\neq 0$ is abelian at all orders. If we fix the multiplicator of the associated Darboux point to $-1$, the Galois group of the first order variational equation has an abelian neutral component if and only if
\[  \Sp\left(\nabla^2 V(c)\right)\subset   \left\lbrace\textstyle\frac{1}{2}(k-1)(k+2):\; k\in\mathbb{N} \right\rbrace. \]
If the multiplicator of the Darboux point is $0$, the Galois group of the first order variational equation has an abelian neutral component if and only if
\[ \Sp\left(\nabla^2 V(c)\right)\subset \left\lbrace 0 \right\rbrace. \]
\end{thm}

In fact, this is not exactly the same statement as the original theorem because we allow $r$ to appear in the potential and in the first integrals.

\begin{proof}
Let $\Gamma\subset \mathbb{C}^2\times \mathcal{S}$ denote the curve defined by equation \eqref{eq.orbit} without the singular point $(q_1,q_2,r)=(0,0,0)$, and $M$ an open neighbourhood of $\Gamma$ in $\mathbb{C}^2\times \mathcal{S}$ such that $H$ is holomorphic on $M$. The Hamiltonian $H$ is then well defined and holomorphic on a symplectic manifold $M$ and the additional first integral is meromorphic on~$M$. Hence, using the main theorem of \cite{2}, the neutral component of the Galois group of the variational equation near $\Gamma$ is abelian at all orders over the base field of meromorphic functions on~$\Gamma$. The variational equation is a hypergeometric equation. In \cite{8}, Kimura classifies Galois groups of hypergeometric equations over the base field $\mathbb{C}(t)$. We can use this classification as the Galois group over the base field $\mathbb{C}(t)$ is the same as over the base field of meromorphic functions because the hypergeometric equation is a Fuchsian equation (see page 73 of \cite{11}). This produces the condition on the spectrum of $\nabla^2 V(c)$. The case of a degenerate Darboux point leads to the variational equation
$$\ddot{X}=\lambda t^{-3} X$$
which is a Bessel equation (after a change of variables). Its Galois group over the field of meromorphic functions in $t$ has not an abelian identity component except if $\lambda=0$.
\end{proof}

Note that in the case of a degenerate Darboux point, we explicitly need that the first integral is meromorphic including $r=0$, as the variational equation is not regular singular at this point. The integrability condition for a non-degenerate Darboux point also holds for a potential meromorphic only on $\mathcal{S}^*=\mathcal{S}\setminus \{r=0\}$ and meromorphic first integrals on $\mathbb{C}^2\times \mathcal{S}^*$.

\section{Main Results}

In this section, we are going to state the main theorems of this article.
The remaining parts of this paper are dedicated to their proofs.

\begin{thm} \label{thm:main1}
Let $V$ be a homogeneous potential of degree $-1$ in the plane. We suppose that $c=(1,0)$ is a Darboux point of $V$ with multiplicator~$-1$. If the variational equation is integrable at order~$3$, then the following conditions are fulfilled
$$\Sp\left(\nabla^2V(c)\right)= \left\{2, \textstyle\frac{1}{2}(p-1)(p+2) \right\} \hbox{ for some } p\in\mathbb{N}.$$
If $p$ is even then
$$\left(\frac{\partial^3 V}{\partial q_1 \partial q_2^2} \right)^{\!2} f_1(p)+ \left( \frac{\partial^3 V}{\partial q_2^3} \right)^{\!2} f_2(p)+ \left( \frac{\partial^4 V}{\partial q_2^4} \right) f_3(p)=0,$$
and if $p$ is odd then
$$\frac{\partial^3 V}{\partial q_2^3} =0 \quad \hbox{and}\quad \left(\frac{\partial^3 V}{\partial q_1 \partial q_2^2}\right)^{\!2} f_1(p)+ \left(\frac{\partial^4 V}{\partial q_2^4}\right)f_3(p) =0,$$
where the functions $f_1,f_2,f_3$ satisfy explicit P-finite recurrences,
i.e., linear recurrences with polynomial coefficients.
\end{thm}

This theorem is a generalization of the criterion given by Yoshida for
homogeneous potentials in the case of degree $-1$ and dimension~$2$. A
similar theorem could be proven in higher dimensions, but the main
problem is that Theorem \ref{thm:main1} is almost inapplicable in this
form. In most cases, it is necessary to study more closely the
expression of the functions $f_1(p),f_2(p),f_3(p)$ to apply it, and for
the moment, because of limitations of computing power, it seems only
possible to do in dimension $2$ (for which the computations are
already tedious).

\begin{thm} \label{thm:main3}
The functions $f_1(2n),f_2(2n),f_3(2n)$ can be written as
\begin{equation}\begin{split}
f_1(2n) & =\epsilon_1(n)\left(\frac{1511011}{67108864n^2}-\frac{1511011}{134217728n^3}+\frac{31731231}{4294967296n^4} \right)\\
f_2(2n) & =\epsilon_2(n)\left(\frac{22665165}{1073741824n^4}-\frac{22665165}{1073741824n^5}+\frac{298125}{4194304n^6} \right)\\
f_3(2n) & =\epsilon_3(n)\left(-\frac{1740684681}{68719476736n^2}+\frac{1740684681}{137438953472n^3}-\frac{2400813907}{68719476736n^4}  \right)
\label{eq3}
\end{split}\end{equation}
with 
\[|\epsilon_i(n)-1| \leq 10^{-5}\;\;\forall n\geq 100.\]
\end{thm}

\noindent
With this, we can apply Theorem~\ref{thm:main1} to some concrete examples:

\begin{thm} \label{thm:main2}
Let $V$ be a potential in the plane expressed in polar coordinates by
\begin{equation} \label{mat_eq}
V(r,\theta)=r^{-1}\left(a+b e^{ i\theta}+c e^{2 i\theta}+d e^{3 i\theta} \right).
\end{equation}
If $V$ is meromorphically integrable, then $V$ belongs to one of the following families
{\setlength{\arraycolsep}{0.15em}
\begin{equation}\begin{array}{rclrclrcl}
V & = & r^{-1}a, &
V & = & r^{-1}\big(a+be^{i\theta}\big), &
V & = & r^{-1}\big( ae^{i\theta}+be^{3i\theta}\big),\\[0.5em]
V & = & r^{-1}\big(a+be^{2i\theta}\big),\quad &
V & = & r^{-1}\big(a+be^{3i\theta}\big),\quad &
V & = & r^{-1}\big( a+be^{i\theta}\big)^3,
\end{array}\end{equation}
}with $a,b\in\mathbb{C}$.

\end{thm}

The first three families have already known additional first integrals \cite{6}, polynomial of degree $1$ or $2$ in $p$. The status of the last three families is unknown. This is not due to an incomplete application of the Morales-Ramis Theorem, but linked to the fact that either they do not possess any Darboux points, or in the last case the only Darboux point is very degenerate and therefore the Morales-Ramis Theorem gives no integrability constraints at any order, as proven in \cite{10}.

In practical problems like Theorem~\ref{thm:main2}, studying
integrability only using the Morales-Ramis criterion is impossible
because of two facts. First we need a Darboux point of our problem;
if we do not have any, the only thing we can do is to try to find an
additional first integral using the direct method of Hietarinta
\cite{6}.

The second problem is the following scenario: inside the family of
potentials given by Theorem~\ref{thm:main2}, there exist submanifolds
in the space of parameters for which the potential possesses only one
Darboux point and the eigenvalue at this Darboux point can be
arbitrarily high. In this case, the higher variational method is
required. But the constraint at order~$2$ does not give sufficient
conditions to conclude, and it is necessary to go to order~$3$.

But the expression of this constraint cannot be written explicitly for
all possible eigenvalues, only for a finite number of them.  To apply
this third-order criterion, we derive P-finite recurrences and
asymptotic expansions with error control in
Theorem~\ref{thm:main3}. This allows us to prove that the
integrability condition is not fulfilled.  The proof of
Theorem~\ref{thm:main2} therefore will be split into two parts:
\begin{enumerate}
\item The first part consists in constructing a manifold
  $M$ in the space of the parameters $a,b,c,d$ such that if the eigenvalues for all Darboux points are real,
  then the parameters belong to~$M$. Then we produce a
  decomposition $M=M_1\cup\dots M_k$ and study each manifold separately. For some of them, the
  corresponding potentials possess sufficiently many Darboux points to give a strong enough condition for integrability only using          
  the Morales-Ramis criterion at order $1$ (there could exist some resistant cases for which a higher variational equation is needed but
  without the phenomenon of arbitrary high eigenvalues like in~\cite{4}). But for specific cases, this phenomenon occurs. It has already
  been noticed by Maciejweski in~\cite{3} who lets this specific case open.
\item The second part will be devoted to these specific manifolds
  $M_i$ where the Morales-Ramis criterion at order~$1$ is almost
  powerless. We use Theorems~\ref{thm:main1}
  and~\ref{thm:main3} to solve these hard cases.
\end{enumerate}

In \cite{14}, the authors deal with a similar difficulty with the spring pendulum for which there is a discrete infinite set of parameters for which there are no obstructions to integrability at order $2$. They also study third order variational equations, but then use analytic tools to study a sequence of monodromy elements, and finally prove that this sequence never vanishes. Thanks to our explicit expression via $P$-finite recurrences, such a problem can be analysed more systematically here.

\section{Eigenvalue Bounding}

\begin{defi} We will denote
$$\mathcal{M}=\left\lbrace V(r,\theta)=r^{-1} U(\theta) \hbox{ with } U \hbox{ meromorphic and } 2\pi\hbox{-periodic} \right\rbrace.$$
Let $V\in\mathcal{M}$. We denote by $d(V)$ the set of Darboux points $c$ of $V$ with multiplicator~$-1$ and $c_3\neq 0$. For $c\in d(V)$ we have $\Sp\left(\nabla^2 V(c)\right)=\{2,\lambda\}$ and we denote
$$\Lambda(c)=\left\lbrace\begin{array}{c} \;\;\lambda\quad\hbox{ if } \lambda\in\mathbb{R} \\ -\infty\;\; \hbox{ otherwise}\\ \end{array}\right..$$
\end{defi}

\begin{defi} We consider a subset $E\subset \mathcal{M}$ and define
$$\Lambda(E)= \sup\limits_{V\in E,\; d(V)\neq \varnothing} \;\inf\limits_{c \in d(V)} \Lambda(c).$$
We say that the problem of finding all meromorphically integrable potentials in $E$ is a bounded eigenvalue problem if $\Lambda(E)<\infty$.
\end{defi}

\begin{rem} We have $\Lambda(\mathcal{M})=\infty$ because of the following family
$$V(r,\theta)=r^{-1}\left((1+a)-2ae^{i\theta}+ae^{2i\theta}\right),\quad a\in\mathbb{R},$$
for which only one Darboux point $c=(1,0)$ exists; the corresponding
eigenvalue is $\lambda=2a-1$. This proves that the family
of potentials considered in Theorem~\ref{thm:main2} is an unbounded
eigenvalue problem.
\end{rem}

\begin{lem}\label{thm:lem1} 
For a potential $V\in\mathcal{M}$ the Darboux points $c$ such that
$c_3 \neq 0$ can be written as 
$c=(c_1,c_2)=(r_0\cos(\theta_0),r_0\sin(\theta_0))$ with $\theta_0$ being a
critical point of~$U$. The Darboux point $c$ is non-degenerate if and
only if $U(\theta_0)\neq 0$ and in this case, the eigenvalues of the
Hessian of $V$, evaluated at $c$, are
\[
  \Sp\left(\nabla^2 V(c)\right)=\left\lbrace 2,\,
  \frac{U''(\theta_0)}{U(\theta_0)} -1 \right\rbrace.
\]
if we choose the multiplicator of $c$ to be~$-1$.
\end{lem}
\begin{proof}
For $V=r^{-1}U(\theta)$ the conditions~\eqref{eq.Darboux} that
$c$ is a Darboux point are:
\begin{align*}
  r_0^{-3}\left(-c_1U(\theta_0)-c_2U'(\theta_0)\right) & = \alpha c_1,\\
  r_0^{-3}\left(-c_2U(\theta_0)+c_1U'(\theta_0)\right) & = \alpha c_2.
\end{align*}
Assuming $c_3\neq0$, it follows that $U(\theta_0)=-\alpha r_0^3$ and
$U'(\theta_0)=0$, which means that $\theta_0$ is a critical point of~$U$.
Since $c_3\neq0$ implies that $r_0\neq0$, we see that the case $\alpha=0$
(degenerate Darboux point) is equivalent to $U(\theta_0)=0$. Setting
$\alpha=-1$ and $U'(\theta_0)=0$ we get the Hessian matrix
\[
  \nabla^2 V(c) = \frac{1}{r_0^5}\left(\!\!\begin{array}{cc}
  (2c_1^2-c_2^2)U(\theta_0)+c_2^2U''(\theta_0) & 
    c_1c_2(3\,U(\theta_0)-U''(\theta_0))\\[0.5em]
  c_1c_2(3\,U(\theta_0)-U''(\theta_0)) & 
    (2c_2^2-c_1^2)U(\theta_0)+c_1^2U''(\theta_0)
  \end{array}\!\!\right)
\]
whose eigenvalues are exactly those claimed above (using 
$U(\theta_0)=r_0^3$).
\end{proof}

Recall that the potentials given by~\eqref{mat_eq} are
$V(r,\theta)=r^{-1}U(\theta)$ with 
$U(\theta)=a+b e^{ i\theta}+c e^{2 i\theta}+d e^{3 i\theta}$.  
We now assume that $V$ possesses at least
one non-degenerate Darboux point~$c$ with $c_3 \neq0$.
After rotation, we can always assume that $c=(1,0)$ is a Darboux
point.  As shown in Lemma~\ref{thm:lem1}, it corresponds to a critical
point for $\theta=0$. Moreover, because this Darboux point is non-degenerate, we know that $U(0)\neq 0$. Then by dilatation, we can
also suppose that $U(0)=1$ and get the following equations
\begin{align*}
U(0) & = a+b+c+d=1,\\
U'(0) & = i(b+2c+3d)=0.
\end{align*}
Solving these equations for $c$ and $d$, yields the expression
\[ 
  V_{a,b}=r^{-1}\left(a+b e^{ i\theta}+(3-3a-2b) e^{2 i\theta}+(2a+b-2)e^{3 i\theta} \right)
\]
for the potentials where $a,b\in\mathbb{C}$.

\begin{thm}\label{thm.families}
If $V_{a,b}$ is meromorphically integrable, then it belongs to one
of the following families
\begin{align*}
E_1 & = r^{-1}\left(-\frac{1}{3}b+1+be^{i\theta}-be^{2i\theta}+\frac{1}{3}be^{3i\theta}\right),\\
E_2 & = r^{-1}\left(-\frac{1}{6}k(k+1)e^{3i\theta}+\frac{1}{4}k(k+1)e^{2i\theta}-
  \frac{1}{12}k^2-\frac{1}{12}k+1\right),\\
E_3 & = r^{-1}\left(-\frac{1}{4}k(k+1)e^{2i\theta}+\frac{1}{2}k(k+1)e^{i\theta}-
  \frac{1}{4}k^2-\frac{1}{4}k+1\right),\\
E_4 & = r^{-1}\left(\frac{(s-6\lambda_2)\lambda_2}{18(\lambda_1+\lambda_2)}e^{3i\theta}-
  \frac{(3\lambda_1+s-3\lambda_2)\lambda_2\phantom{^2}}{6(\lambda_1+\lambda_2)}e^{2i\theta}+\right.,\\
& \qquad\qquad \left.\frac{(6\lambda_1+s)\lambda_2}{6(\lambda_1+\lambda_2)}e^{i\theta}+
  \frac{-9\lambda_1\lambda_2-\lambda_2s+18\lambda_1+18\lambda_2-3\lambda_2^2}{18(\lambda_1+\lambda_2)}\right)
\end{align*}
where $b\in\mathbb{C}$ and $k\in\mathbb{N}$. The quantities arising in $E_4$ are
\begin{align*}
s^2 & = 6\lambda_1^2\lambda_2+6\lambda_1\lambda_2^2-36\lambda_1\lambda_2,\\
\lambda_i & =\frac{1}{2}(k_i-1)(k_i+2)+1\quad (i=1,2),
\end{align*}
with $k_1\in\mathbb{N}\setminus\{0,3\}$ and $k_2\in\mathbb{N}^*$.
\end{thm}
\begin{proof} 
For all non-degenerate Darboux points
$c=(\gamma \cos(\theta_0),\gamma \sin(\theta_0))$
the corresponding eigenvalue~$\lambda$ satisfies 
\begin{equation}\label{eq.condEV}
  U''(\theta_0)-(\lambda+1) U(\theta_0)=0\quad\text{and}\quad U'(\theta_0)=0
\end{equation}
(note that this condition is also satisfied if $c$ is degenerate). 
We write $U(\theta)=h_{a,b}(\exp(i\theta))$, 
$U'(\theta)=izh'_{a,b}(\exp(i\theta))$,
and $U''(\theta)=\tilde{h}_{a,b}(\exp(i\theta))$ with
\begin{align*}
h_{a,b}(z) & = a+bz+(3-3a-2b)z^2+(2a+b-2)z^3,\\
\tilde{h}_{a,b}(z) & = -bz-4(3-3a-2b)z^2-9(2a+b-2)z^3.
\end{align*}
So to find the eigenvalues of all Darboux points, one just needs to
compute the following resultant which corresponds to the
conditions~\eqref{eq.condEV}:
\begin{align*}
P_{a,b}(\lambda) & = \res_z\!\left(\tilde{h}_{a,b}(z)-(\lambda+1)h_{a,b}(z),\,h'_{a,b}(z)\right)\\
& = (2a+b-2)(6a+2b-6+(\lambda+1))(-18ab^2-6b^3+18b^2+(\lambda+1)\\
&\quad \times(108a^3+108a^2b-216a^2+36ab^2+108a-108ab-9b^2+4b^3)) 
\end{align*}

All the roots of $P_{a,b}(\lambda)$ correspond to an eigenvalue of
some Darboux point, except possibly in those cases $(a,b)$ where
$P_{a,b}$ vanishes as a polynomial in~$\lambda$ or in the case where
$h'_{a,b}(z)$ has the root~$0$.

Let us begin with the special cases. We compute the points
$(a,b)\in\mathbb{C}^2$ for which $P_{a,b}=0$ in
$\mathbb{C}[\lambda]$. We find that it is the zero set of
the ideal $\langle 2a+b-2\rangle \cap \langle a,b\rangle$.
Moreover, the polynomial $h'_{a,b}(z)$ has a zero root if and only if
$b=0$. So, all the specific cases belong to the zero set of
$\langle 2a+b-2\rangle \cap \langle b\rangle$.
First, for $b=0$ we find
\begin{align*}
Q_1 & = \res_z(\tilde{h}_{a,0}(z)-(\lambda+1)h_{a,0}(z),\,h'_{a,0}(z)/z,z)\\
    & = 216(a-1)^3(6a-6+(\lambda+1)),
\end{align*}
and second, for $b=2-2a$ we get
\begin{align*}
Q_2 & = \res_z(\tilde{h}_{a,2-2a}(z)-(\lambda+1)h_{a,2-2a}(z),\,h'_{a,2-2a}(z),z)\\
    & = -4(a-1)^2(2a-2+(\lambda+1)).
\end{align*}
As we know that the eigenvalues should be of the form
$\frac{1}{2}(k-1)(k+2)$, $k\in\mathbb{N}$, we obtain the potentials $E_2$
and $E_3$ from these two cases.

Now for the generic case, we express $a$ and $b$ depending on the
roots of $P_{a,b}(\lambda)$ and obtain the expression~$E_4$.  
Since it is not valid for $k_1=k_2=0$, we study this case
separately and find the condition $a=-\frac{1}{3}b+1$, which 
gives~$E_1$.  Note that fixing $\lambda_1=0$ in $E_4$ yields the
potential~$E_2$, whereas $\lambda_1=6$ results in~$E_3$. The case 
$k_2=0$ produces $V=r^{-1}$ which already belongs to $E_1$.
\end{proof}

\begin{cor} With the same notation as in Theorem~\ref{thm.families}, 
we have $\Lambda(E_1)=-1$ and 
$\Lambda(E_2)=\Lambda(E_3)=\Lambda(E_4)=\infty$.
\end{cor}

\begin{rem} 
The types of $E_2$, $E_3$ and $E_4$ differ fundamentally although they
are all unbounded eigenvalue problems. This is because the dimension
of $E_4$ is $2$ and the dimension of $E_2$ and $E_3$ is
only~$1$. Because of that, we could call $E_4$ a doubly unbounded
eigenvalue problem because it possesses two Darboux points whose
eigenvalues can be independently arbitrarily high. Because of that, we
will need to apply a third order integrability criterion
simultaneously at the two Darboux points. The potential $E_1$ has only
one Darboux point with eigenvalue~$-1$. This eigenvalue belongs to the
Morales-Ramis table and so higher variational methods will be
required, but only for this fixed eigenvalue (which is much easier).
\end{rem}

In the parameter space, we get $4$ algebraic
mani\-folds. For $E_2$, $E_3$, and $E_4$, a tedious treatment with
higher variational equations is required. For $E_1$ we will be able to
check integrability easily with Theorem~\ref{thm:main1}. A similar
procedure could be applied to any set of homogeneous potentials
depending rationally on some parameters. Here computing power is the
main limitation; in particular, because for typical problems, the
number of parameters is much smaller than the number of roots which
requires resultant computations and prime ideal
decompositions. One should note that we have deliberately chosen a set
of potentials~\eqref{mat_eq} which is particularly difficult to
treat. For most common problems (outside the general complete
classification), these unbounded eigenvalue manifolds have small
dimension ($1$ in the case found by~\cite{3}) or even inexistent like
in~\cite{4} or~\cite{5}.

\section{Higher Order Variational Methods}
We will first recall some properties of the solutions of the first
order variational equations. After diagonalisation and in the
integrable case, the equation is the following (after fixing the
energy $E=1$)
\begin{equation}\label{eq1}
2t^2(1+t)y''(t) - ty'(t) -\textstyle{\frac{1}{2}}(n-1)(n+2)y(t) = 0\qquad (n\in\mathbb{N}).
\end{equation}
After the change of variables $t\longrightarrow (t^2-1)^{-1}$,
this equation becomes
\begin{equation}\label{eq.diffeqPQ}
  (t^2-1) y''(t) + 4ty'(t) - (n-1)(n+2) y(t)=0\qquad (n\in\mathbb{N}).
\end{equation}
A basis of solutions is given by $(P_n,Q_n)$ where $P_n$ are 
polynomials in $t$ (for $n\geq 1$) and the functions $Q_n$ are
\[
  Q_n(t)=P_n(t)\int \frac{1}{(t^2-1)^2P_n(t)^2} dt.
\]
The functions~$Q_n$ are multivalued except for $n=0$ which will be a
special case. Indeed, the Galois group of \eqref{eq1} in this case is
$I\!d$ instead of $\mathbb{C}$.

The polynomials~$P_n$ can be computed using the Rodrigues type formula
\begin{equation}\label{eq.P}
  P_n(t)=\frac{1}{t^2-1}\frac{\partial^{n-1}}{\partial t^{n-1}} (t^2-1)^n\qquad (n\geq 1)
\end{equation}
which also gives a normalisation for their leading coefficient. 
The functions $Q_n$ can be written as
\begin{equation}\label{eq.Q}
  Q_n(t)=\epsilon_n P_n(t) \arctanh\left(\frac{1}{t}\right)+\frac{W_n(t)}{t^2-1}\qquad (n\geq 1)
\end{equation}
where $W_n$ are polynomials given by
\begin{align*}
  W_{2k}(t) & = \frac{(-1)^k(t^2-1)}{2^{4k}}\left(
    \frac{\pi\ {}_2F_1\big(\frac12-k,k+1,\frac12,t^2\big)}{\Gamma\big(k+\frac12\big)^2}+{}\right.\\
           & \quad\left.
    \frac{2kt(2k+1)\arctanh(t)\ {}_2F_1\big(1-k,k+\frac32,\frac32,t^2\big)}{(k!)^2}\right),\\
  W_{2k+1}(t) & = \frac{(-1)^k(t^2-1)}{2^{4k+2}}\left(
    \frac{\pi t(k+1)(2k+1){}_2F_1(\frac12-k,k+2,\frac32,t^2)}{\Gamma(k+\frac32)^2}-{}\right.\\
  & \quad\left.\frac{2\arctanh(t)\ {}_2F_1(-k,k+\frac32,\frac12,t^2)}{(k!)^2}\right)
\end{align*}
and $\epsilon_n$ is a real sequence given by
\[
  \epsilon_n=\frac{n(n+1)}{4^n(n!)^2}.
\]
Conventionally, we will take for $n=0$:
$$P_0(t)=\frac{t}{t^2-1}, \qquad Q_0(t)=\frac{1}{t^2-1}.$$
%In fact, this also satisfy the formulas before, considering that the derivation -1 is an integration from 0.

\begin{lem}
The functions $P_n(t)$ and $\frac{1}{\epsilon_n} Q_n(t)$ satisfy the
differential equation~\eqref{eq.diffeqPQ} and the three-term
recurrence
\[
  4n(n+1)(n+2)y_n(t) - 2t(n+2)(2n+3)y_{n+1}(t) + (n+3)y_{n+2}(t) = 0.
\]
\end{lem}
\begin{proof}
Given the explicit expressions~\eqref{eq.P} and~\eqref{eq.Q} we can
use holonomic closure properties to derive the differential equation
resp. recurrence they satisfy. We first express~\eqref{eq.P} as
\[
  P_n(t)=\frac{(n-1)!}{2\pi i(t^2-1)} \oint\frac{(u^2-1)^n}{(u-t)^n}\, du
\]
by Cauchy's differentiation formula. By the method of creative 
telescoping we obtain the differential equation and the recurrence
(this calculation was carried out by the software package 
\texttt{HolonomicFunctions}~\cite{Koutschan09,Koutschan10b}).
Similarly we can apply holonomic closure properties to the closed
form expression~\eqref{eq.Q}.
\end{proof}

\begin{lem} \label{lem:main1} (proved in \cite{9})
Let $F\in\mathbb{C}(z_1)\left[z_2 \right]$ and
$f(t)=F\big(t,\arctanh\!\big(\frac{1}{t}\big)\big)$.
We consider the field extension
\[  K=\mathbb{C}\left(t,\arctanh\left(\frac{1}{t}\right),\int f\, dt \right) \]
and the monodromy group $G=\sigma(K,\mathbb{C}(t))$.
If $G$ is abelian, then
$$\frac{\partial}{\partial \alpha} \Res_{t=\infty} F\left(t,\arctanh\left(\frac{1}{t}\right)+\alpha\right) =0.$$
\end{lem}

\begin{proof} We will consider two paths, the ``eight'' path $\sigma_1$ around the singularities $-1$ and $1$, and the path~$\sigma_2$ around infinity. At infinity, the function $F\big(t,\arctanh\!\big(\frac{1}{t}\big)+\alpha\big)$ will have a series expansion of the kind
$$\int F\left(t,\arctanh\left(\frac{1}{t}\right)+\alpha\right) dt=\sum\limits_{n=n_0}^{\infty} a_n(\alpha)t^n+r(\alpha)\ln t$$
because the function $\arctanh\!\big(\frac{1}{t}\big)$ has a regular point at infinity. Let us now consider the monodromy commutator
$$\sigma=\sigma_2^{-1}\sigma_1^{-\frac{\beta}{2i\pi}}\sigma_2\sigma_1^{\frac{\beta}{2i\pi}} \qquad \beta\in 2i\pi \mathbb{Z}.$$
We have that $\sigma_1^{\frac{\beta}{2i\pi}}(f)=F\big(t,\arctanh\!\big(\frac{1}{t}\big)+\beta\big)$ and $\sigma_2(\ln t)=\ln t +2i\pi$. We deduce that
$$\sigma(f)=f+r(\beta)-r(0).$$
This $r(\beta)$ corresponds to the residue of
$F\big(t,\arctanh\!\big(\frac{1}{t}\big)+\beta\big)$ at
infinity. If the monodromy is commutative, then the commutator
$\sigma$ should act trivially on~$f$. This is the case only if
$r(\beta)-r(0)=0$ for all $\beta\in 2i\pi\mathbb{Z}$. The function $r$ is a
polynomial in $\beta$, so $r(\beta)-r(0)=0$ for all
$\beta\in\mathbb{C}$.  From this the claim follows.
\end{proof}

In the following, we will also need to use the next lemma which
is a kind of reciprocal version of Lemma~\ref{lem:main1}.

\begin{lem}\label{lem:main1reciproque} (proved in \cite{9})
We consider
$$F(t)=\sum\limits_{i=0}^3 H_i(t)  \arctanh\left(\frac{1}{t}\right)^i$$
with $H_0,\dots,H_3\in\mathbb{C}[t]$. If the conditions of Lemma~\ref{lem:main1} are satisfied, then
\begin{itemize}
 \item If $\Res_{t=\infty} F(t) =0$, then $\int F\,dt \in\mathbb{C}\left[t,\arctanh\left(\frac{1}{t}\right) \right]$
 \item If $\Res_{t=\infty} F(t) \neq 0$, then $\int F\,dt \in\mathbb{C}\left[t,\arctanh\left(\frac{1}{t}\right),\ln\left( t^2-1 \right) \right]$
\end{itemize}
\end{lem}

\begin{thm} \label{thm:second1}
Let $V$ be a homogeneous potential of degree $-1$ in the plane. We
suppose that $c=(1,0)$ is a Darboux point of $V$ with
multiplicator~$-1$. If the variational equation is integrable at
order~$2$ then
\[
  \Sp\left(\nabla^2V(c)\right)= \left\{2, \textstyle{\frac{1}{2}}(p-1)(p+2) \right\},\quad 
  p\in\mathbb{N},
\]
and for odd $p$ we have $\frac{\partial^3 V}{\partial q_2^3} =0$.
\end{thm}
This theorem is in fact a particular case of Theorem $2$ in \cite{9}
for which the three indices $i,j,k$ are equal.

\begin{rem}
Because the constraint appears only for odd $p$, the variational equations of order~$2$ give no constraint for even~$p$. Hence this is not sufficient for proving non-integrability for an unbounded manifold.
\end{rem}

\section{Proof of Theorem~\ref{thm:main1}}\label{sec}
\begin{proof}
The variational equation at order $3$ is given by
\begin{align*}
\ddot{X}_1 & =\frac{2}{\phi^3} X_1 + \frac{1}{2}\frac{a}{\phi^4} Y_{1,1}-\frac{4b}{3\phi^5}Z^3\\
\ddot{X}_2 & =\frac{\lambda}{\phi^3} X_2 +\frac{a}{\phi^4} Y_{2,1}+ \frac{b}{\phi^4} Y_{1,1}+\frac{c}{\phi^5} Z^3\\
\dot{Y}_{1,1} & =2Y_{1,2}\\
\dot{Y}_{1,2} & =\frac{\lambda}{\phi^3}Y_{1,1}+\frac{b}{\phi^4} Z^3+Y_{1,3}\\
\dot{Y}_{1,3} & =\frac{\lambda}{\phi^3}Y_{1,2}+\frac{b}{\phi^4} Z^2\dot{Z} \\
\dot{Y}_{2,1} & =Y_{2,2}+Y_{2,3}\\
\dot{Y}_{2,2} & =\frac{2}{\phi^3}Y_{2,1}-\frac{4b}{3\phi^5}Z^3+Y_{2,4}\\
\dot{Y}_{2,3} & =\frac{\lambda}{\phi^3}Y_{2,1}+Y_{2,4}\\
\dot{Y}_{2,4} & =\frac{2}{\phi^3}Y_{2,3}-\frac{4b}{3\phi^5}Z^2\dot{Z}+\frac{\lambda}{\phi^3}Y_{2,2}\\
\ddot{Z} & =\frac{2}{\phi^3} Z
\end{align*}
where $\lambda=\frac{1}{2}(n-1)(n+2)$. The coefficients $a,b,c$ correspond to the following derivatives
$$a=\frac{\partial^3}{\partial q_1 \partial q_2^2} V(c), \quad b=\frac{1}{2} \frac{\partial^3}{\partial q_2^3} V(c), \quad c=\frac{1}{6}\frac{\partial^4}{\partial q_2^4} V(c),$$
and the others are given using the Euler relation for homogeneous functions. A complete procedure to build these equations is given by~\cite{7}. The functions $Y_{1,1}$ and $Y_{2,1}$ are solutions of a system of linear differential equations with an inhomogeneous term, and the homogeneous part is in fact a symmetric product of the first order variational equation. Here, we already put to zero terms that we think in advance they will not produce integrability constraints. As before, we use the change of variables $\phi(t)\longrightarrow (t^2-1)^{-1}$.

We choose $Z(t)=Q_n$ and compute the solution for $X_2$ of the above system. We first remark that $X_2$ is in the Picard-Vessiot field, so it is also the case for its derivative. We now perform integration by parts and see that one term is already in the Picard-Vessiot field, and the other is
\begin{equation}\label{eq.intPV}
  \int 2(t^2-1)^2\left(a^2tP_nQ_nI_1 + 4b^2P_n^2Q_nI_2 + c(t^2-1)Q_n^4 \right) dt
\end{equation}
where
\begin{align*}
I_1 & = \int\left(
  \frac{\int\left(\frac{t(t^2-1)^2Q_n^3}{P_n}+\frac{I_3}{(t^2-1)^2P_n^2}\right)dt}{t^2(t^2-1)^2} +
  \frac{\int\frac{I_3}{t^2(t^2-1)^2}dt}{(t^2-1)^2P_n^2}\right) dt\\
I_2 & = \int\frac{\int\left(
    (t^2-1)^2Q_n^3 + 
    \frac{2}{(t^2-1)^2 P_n^2}\int (t^2-1)^4P_nQ_n^2(P_n\dot{Q}_n-Q_n\dot{P}_n)dt
  \right)dt}{(t^2-1)^2P_n^2}dt\\
I_3 & = \int t(t^2-1)^4Q_n^2\left(P_n\dot{Q}_n-Q_n\dot{P}_n\right)dt
\end{align*}

Let us now study this expression term by term. We begin with the 
third summand of~\eqref{eq.intPV} which is
$$2c\int \big(t^2-1\big)^3Q_n^4\,dt.$$
It has already the form of Lemma~\ref{lem:main1}. So as in the proof of Lemma~\ref{lem:main1}, the monodromy commutator will be computed using
$$\Res_{t=\infty} (t^2-1)^3 (Q_n+\epsilon_n\alpha P_n)^4.$$
Now look at the term in $b^2$. It is not as complicated as we could think because of the following relation
$$P_n\dot{Q}_n-\dot{P_n} Q_n=(t^2-1)^{-2}\quad \forall n\in\mathbb{N}$$
which is linked to the Wronskian of Equation~\eqref{eq.diffeqPQ}.
Thanks to that, the term in $b^2$ can be written as
$$8b^2\int P_n^2 Q_n (t^2-1)^2 \int \frac{\int (t^2-1)^2 Q_n^3+2 {\frac {\int P_nQ_n^2(t^2-1)^2 dt}{(t^2-1)^2P_n^2}}dt}{(t^2-1)^2P_n^2} dt\,dt$$
and then using integration by parts, this gives
$$16b^2\int Q_n^3 (t^2-1)^2\int P_nQ_n^2(t^2-1)^2 dt\,dt-8b^2\int P_n Q_n^2 (t^2-1)^2 dt \int Q_n^3(t^2-1)^2 dt$$ %I propose to just keep this final line of the computation
Now by Lemma~\ref{lem:main1reciproque} we have for all 
even integers~$n>1$:
\[
  \int P_nQ_n^2(t^2-1)^2 dt,\;\int Q_n^3 (t^2-1)^2 dt 
  \in\C(t)\left[ \arctanh\left(\frac{1}{t}\right) \right].
\]
So we are integrating a polynomial in $\arctanh$ with rational coefficients, and this corresponds to the hypotheses of Lemma~\ref{lem:main1}. The second term does not provide any monodromy, so we only have to study the first term and thus the sequence
$$\Res_{t=\infty} (Q_n+\epsilon_n\alpha P_n)^3 (t^2-1)^2\int P_n(Q_n+\epsilon_n\alpha P_n)^2(t^2-1)^2.$$
Now we look at the term in $a^2$. It can be simplified to
$$\int 2a^2 (t^2-1)^2P_n Q_n t \int \frac{\int \!\frac {(t^2-1)^2 Q_n^3 t} {P_n}+{\frac{\int \! (t^2-1)^2 Q_n^2 t\,dt}{(t^2-1)^2 P_n^2}}{dt}} {t^2(t^2-1)^2}+\frac{\int \!\frac {\int \! (t^2-1)^2 Q_n^2t\,dt} {t^2(t^2-1)^2}dt}{(t^2-1)^2P_n^2}dt$$
We now use again integrations by parts (recall that $P_2=4t$):
$$8a^2\int \!(t^2-1)^2 Q_n^2 Q_2\int \! (t^2-1)^2 Q_n^2 t\,dt-8a^2\int\! (t^2-1)^2Q_n^2 t \int \!(t^2-1)^2 Q_n^2Q_2 dt.$$ %I propose to just keep this final line of the computation
To conclude we can again use Lemmas~\ref{lem:main1} and~\ref{lem:main1reciproque}. We first prove that
$$\forall n\neq 1\;\;\int P_2Q_n^2(t^2-1)^2 dt,\;\int Q_n^2 Q_2 (t^2-1)^2 dt \in\C(t)\left[ \arctanh\left(\frac{1}{t}\right) \right]. $$ %It is necessary to compute the residues of Lemma~\ref{lem:main1} to prove this
The case $n=1$ corresponds to $\lambda=0$, for which we have always the coefficient $a=0$. Now we make a final integration by parts which gives
$$16a^2\int \!(t^2-1)^2 Q_n^2 Q_2\int \! (t^2-1)^2 Q_n^2 t\,dt\,dt-8a^2\int\! (t^2-1)^2Q_n^2 t\,dt \int \!(t^2-1)^2 Q_n^2Q_2 dt.$$
Thanks to that, we get a constraint of the form given by
Theorem~\ref{thm:main1} and the coefficients are given by (multiplying them by $\epsilon_n^{-2}$ for further simplifications)
\begin{align}
\label{eq.f1}
f_1(n) & = \langle\alpha^3\rangle\, 2\epsilon_n^{-2}\Res_{t=\infty}\bigg(
    (t^2-1)^2 (Q_n+\epsilon_n\alpha P_n)^2 (Q_2+\epsilon_2\alpha P_2)\\
  &\qquad \times\int(t^2-1)^2(Q_n+\epsilon_n\alpha P_n)^2P_2\,dt\bigg),\nonumber\\
\label{eq.f2}
f_2(n) & = \langle\alpha^3\rangle\, 2\epsilon_n^{-2}\Res_{t=\infty} \bigg( 
    (t^2-1)^2(Q_n+\epsilon_n\alpha P_n)^3\\
  &\qquad \times\int (t^2-1)^2(Q_n+\epsilon_n\alpha P_n)^2P_n\,dt\bigg),\nonumber\\
\label{eq.f3}
f_3(n) & = \langle\alpha^3\rangle\, \frac{1}{6}\epsilon_n^{-2}\Res_{t=\infty}\left( 
  (t^2-1)^3 (Q_n+\epsilon_n\alpha P_n)^4\right),
\end{align}
where $\langle\cdot\rangle$ denotes coefficient extraction. In fact,
only the coefficient of~$\alpha^3$ appears in these residues. We need
not to prove this fact, because we simply select the coefficient
of~$\alpha^3$, ignoring the question whether the other coefficients
are zero or not.

We now look at the case $n=0$. All our previous calculations are also valid in this case except those involving Lemma~\ref{lem:main1reciproque} because we only have
$$\int P_0Q_0^2(t^2-1)^2 dt,\;\int Q_0^3 (t^2-1)^2 dt \in\C(t)\left[ \arctanh\left(\frac{1}{t}\right),\ln(t^2-1) \right], $$
$$\int P_2Q_0^2(t^2-1)^2 dt,\;\int Q_0^2 Q_2 (t^2-1)^2 dt \in\C(t)\left[ \arctanh\left(\frac{1}{t}\right) \right]. $$
So, the coefficients in $a^2,c$ are also
\begin{align*}
2\Res_{t=\infty} \left( (t^2-1)^2 Q_0^2 Q_2\int (t^2-1)^2P_2Q_0^2\,dt\right),\\
\frac{1}{6}\Res_{t=\infty} \left( (t^2-1)^3 Q_0^4\right).
\end{align*}
We find that these residues are both~$0$, and so the corresponding integral does not provide any additional monodromy. The case of the coefficient in $b^2$ is a little more difficult because the integral does not satisfy the conditions of Lemma~\ref{lem:main1}. After an explicit computation, we arrive at the following integral
\begin{align*}
\int \frac{1}{t^2-1}\left(-t \arctanh\left(\frac{1}{t}\right)-\frac{1}{2} \ln\left(t^2-1\right) \right)dt=\\
\frac{1}{2}\,\ln  \left( 2 \right) \ln  \left( t-1 \right) +\frac{1}{2}\,\dilog \left( t+1 \right)+\\
\frac{1}{8} \ln  \left( t+1 \right)^2+ \frac{1}{4}\ln  \left( t+1 \right)\ln  \left( t-1 \right)-\frac{1}{8} \ln  \left( t-1 \right)^2.
\end{align*}
All the terms are in $\mathbb{C}[t,\arctanh\left( \frac{1}{t} \right),\ln\left( t^2-1 \right) ]$ except one, namely the dilogarithmic term
$$\dilog\left( t+1 \right)=\int \frac{\ln(t+1)}{t}dt.$$
With the same idea as in Lemma~\ref{lem:main1}, we see that this term has a noncommutative monodromy because of the following residue in $0$
$$\Res_{t=0} \frac{\ln(t+1)+\alpha}{t}=\alpha$$
which depends explicitly on~$\alpha$. So, for $n=0$, the integrability condition at order~$3$ is in fact just $b^2=0$.

\end{proof}

% \begin{center}
% List of the $3$th order integrability constraint for the first values of $n$
% \end{center}
% $$n=1 \quad -\frac{16}{315}a^2+\frac{8}{35}c=0$$
% $$n=2 \quad -\frac{128}{3465}a^2-\frac{512}{3465}b^2+\frac{128}{385}c=0$$
% $$n=3 \quad -\frac{256}{715}a^2+\frac{13824}{5005}c=0$$
% $$n=4 \quad -\frac{475136}{57057}a^2-\frac{753664}{101745}b^2+\frac{19759104}{323323}c=0$$
% $$n=5 \quad -\frac{2755788800}{7436429}a^2+\frac{19729612800}{7436429}c=0$$
% $$n=6 \quad -\frac{28915531776}{1062347}a^2-\frac{180759822336}{16900975}b^2+\frac{204013043712}{1062347}c=0$$

\section{Holonomicity and Asymptotics}

In this section we are going to derive P-finite recurrences (i.e.,
linear recurrences with polynomial coefficients) for the sequences
$f_1(n)$, $f_2(n)$, and $f_3(n)$ that appeared in section \ref{sec}. The methods that we employ are based on Zeilberger's
holonomic systems approach~\cite{Zeilberger90}. The recurrences
presented below were computed with the method of creative telescoping,
to which a brief introduction is given below (see~\cite{Koutschan09}
for more details).

Let $S_n$ denote the forward shift operator in~$n$, i.e.,
$S_nf(n)=f(n+1)$, and $D_x$ the derivative w.r.t.~$x$, i.e.,
$D_xf(x)=f'(x)$. The method works for the class of holonomic
functions, which in short are (multivariate) functions that are
solutions of maximally overdetermined systems of linear difference and
differential equations with polynomial coefficients. The set of all
equations which a given holonomic function satisfies forms a left ideal
(we call it \emph{annihilating ideal}) in some Ore algebra of the form
\[
  \mathbb{C}(m,n,\dots,x,y,\dots)\langle S_m,S_n,\dots,D_x,D_y\dots\rangle .
\]
The nice fact about holonomic functions is that this class is closed
under certain operations (addition, multiplication, certain
substitutions, definite summation and integration) which can be
executed algorithmically: given the defining systems of equations for
two holonomic functions $f$ and $g$, there are algorithms to compute
a holonomic system for $f+g$, $f\cdot g$, etc.

For computing integrals (or residues), the method of creative
telescoping makes use of the fundamental theorem of calculus. Consider
a definite integral of the form $\int_a^b f\,dx$ where the
integrand~$f$ depends also on some other (discrete and/or continuous)
parameters. We need $f$ to be holonomic, i.e., there is some left
ideal~$I$ of annihilating operators in the corresponding Ore
algebra~$\mathbb{O}$. The idea is now to come up with an operator
$A+D_xB\in I$ where $A,B\in\mathbb{O}$ and~$A$ does not depend on $x$
and $D_x$ (the concept of Gr\"obner bases~\cite{Buchberger65} plays a
crucial r\^{o}le in this step). Then after integration we get,
\[
  P\int_a^b \!f\,dx + \Big[Qf\Big]_a^b = 0,
\]
in other words, we found a (possibly inhomogeneous) equation for the
integral in question. The examples below will demonstrate this
methodology clearly; we start with the simplest one, the
sequence~$f_3(n)$.

\begin{lem}\label{lem.f3}
The sequence $f_3(n)$ given 
in~\eqref{eq.f3}, satisfies the P-finite recurrence
\begin{align*}
& (4n+11)(4n+9)(n+1)^3(n+3)^2f_3(n+2)-{}\\
& (2n+3)(16n^6+144n^5+515n^4+930n^3+888n^2+423n+81)f_3(n+1)+{}\\
& (4n+3)(4n+1)(n+2)^3n^2f_3(n)=0.
\end{align*}
subject to the initial conditions
\[
  f_3(1) = -\frac{8}{105},\quad f_3(2) = -\frac{8}{385}.
\]
\end{lem}
\begin{proof}
It is an easy exercise to compute the first values of $f_3(n)$
explicitly with a computer algebra system. Thus we basically have to
derive the recurrence.  For this purpose, we compute an annihilating
ideal~$I$ for $(t^2-1)^3 (Q_n+\epsilon_n\alpha P_n)^4$ which is the
expression in the residue~\eqref{eq.f3}. For this purpose we apply
holonomic closure properties (note that $Q_n+\epsilon_n\alpha P_n$
satisfies the same equations as $Q_n$ itself). The resulting Gr\"obner
basis is too large to be printed here, namely a full page of equations
approximately. It is represented in the Ore algebra
$\mathbb{C}(n,t)\langle S_n,D_t\rangle$. In the next step we make use
of a special algorithm~\cite{Koutschan10c} for computing a creative
telescoping operator
\[
  A(n,S_n) + D_tB(n,t,S_n,D_t) \in I
\]
(its existence is guaranteed by the theory of holonomy).  Because we
are dealing with a residue we can forget about the part~$B$ and find
that $A$ annihilates the residue. In order to obtain $f_3(n)$ we need
to multiply the residue with $2\epsilon_n^{-2}$, which can be done
again by closure properties. The resulting operator represents exactly
the above recurrence.  All these computations were done with the above
mentioned package
\texttt{HolonomicFunctions}~\cite{Koutschan09,Koutschan10b}.
\end{proof}

\begin{lem}\label{lem.f1}
The sequence $f_1(n)$ given in~\eqref{eq.f1} satisfies the P-finite
recurrence
\begin{align*}
& (4n+11)(4n+9)(n+4)^2(n+1)^3(4n^2+8n-9)f_1(n+2)-{}\\
& (2n+3)(64n^8+768n^7+3580n^6+8028n^5+8113n^4+{}\\
&\quad 834n^3-4863n^2-3276n-648)f_1(n+1)+{}\\
& (4n+3)(4n+1)(n+2)^3(n-1)^2(4n^2+16n+3)f_1(n)=0
\end{align*}
subject to the initial conditions
\[
  f_1(2)=\frac{16}{1155},\quad f_1(3)=\frac{16}{2145}. 
\]
\end{lem}
\begin{proof}
The proof is based on the same ideas as in Lemma~\ref{lem.f3}, except
that the expression of which we have to take the residue is more
complicated.  In particular, an indefinite integral occurs (recall
that indefinite integration is not among the holonomic closure
properties) and it is not clear a priori how to choose the integration
constant such that the result is again holonomic.  We start by
computing an annihilating ideal~$I$ for
\[
  F(n,t)=(t^2-1)^2(Q_n+\epsilon_n\alpha P_n)^2.
\]
Thus for all $A\in I$ the operator $AD_t$ annihilates the indefinite
integral $\int F(n,t)\,dt$. Additionally, from a creative telescoping
operator $A+D_tB\in I$ we can derive more such annihilating operators.
Let $J$ denote the annihilating ideal for $B(F)$ which can be obtained
by holonomic closure properties. Then for every $C\in J$, the
operator~$CA$ annihilates the indefinite integral as well. Altogether
we obtain a zero-dimensional annihilating ideal for $\int F(n,t)\,dt$,
and continue as in Lemma~\ref{lem.f3}.
\end{proof}

These recurrences in Lemmas~\ref{lem.f3} and~\ref{lem.f1} are
irreducible (in the sense that the corresponding operator cannot be
factorized), and so we are not able to find closed forms for $f_1$ and
$f_3$. The recurrence for $f_2(2n)$ is given by a third-order
recurrence with polynomial coefficients of degree larger than~$50$,
which we do not state here explicitly. The initial conditions are
\[
  f_2(2)=\frac{16}{1155},\quad f_2(4)=\frac{184}{183141},\quad f_2(6)=\frac{38308}{181081875}. 
\] 
This recurrence is reducible and possesses a hypergeometric solution
\[  
f_2(2)\frac{8\pi^2\GAMMA(n+1)\GAMMA(5/6+n)^2\GAMMA(1/6+n)^2\GAMMA(n)^3}
{25\GAMMA(n+2/3)^2\GAMMA(3/2+n)^3\GAMMA(1/2+n)\GAMMA(4/3+n)^2}
\]
but because $f_2(2)\neq 0$, the recurrence for $f_2(2n)$ cannot be reduced.

We are interested in a \textbf{practical} way to apply the third-order
variational equation. To do this, these recurrences are not enough,
since we need closed forms. As these closed forms do not exist, we
will instead produce closed form expressions which approach $f_1$,
$f_2$, and $f_3$ with a controlled relative error. In the following,
we will denote the harmonic numbers by
\[  H(n)= \sum\limits_{i=1}^{n-1}\frac{1}{i}. \]

\begin{defi}
Let us consider an operator $L\in\mathbb{C}\langle n,S_n\rangle$,
in other words $L$ represents a linear recurrence with polynomial
coefficients.  We will say that $L$ is regular at infinity if for all
solutions~$u$ (i.e., $Lu=0$) there exist $\alpha\in\mathbb{Z}$,
$\beta\in\mathbb{N}$, and $\gamma\in\mathbb{C}$ such that
\[
  u(n)\sim \gamma n^\alpha H(n)^\beta  \quad\text{for } n\to\infty.
\]
\end{defi}

\begin{thm}
Consider $L\in\mathbb{C}\langle n,S_n\rangle$ of order $k$
and assume that it is regular at infinity. Then for all $p\in\mathbb{N}$
and for all $u$ solution of $Lu=0$, there exists a function
$F\in\mathbb{C}(n)[H(n)]$ with degree in $H(n)$ less than $k-1$
such that
$$u(n)=F(n) +O\left( \frac{H(n)^{k-1}}{n^p} \right).$$
\end{thm}

This theorem is directly implied by the theorem of Birkoff given in \cite{13}, which gives a general form of an asymptotic expansion which is always possible. In our case, we will only use what we call the regular case, which in a Birkoff expansion corresponds to not having an exponential part.

\begin{defi}
Consider a function $f:\mathbb{N} \longrightarrow \mathbb{R}$ and a function $F\in\mathbb{R}(n)[H(n)]$. We say that $F$ is an approximation of $f$ with relative error $\epsilon$ at rank $n_0$ if
$$\left| \frac{f(n)}{F(n)} -1 \right| \leq \epsilon \quad \forall n\geq n_0.$$ 
We consider $p$ functions $f_1,\dots,f_p:\mathbb{N} \longrightarrow \mathbb{R}$ and approximations $F_1,\dots,$ $F_p\in\mathbb{R}(n)[H(n)]$ with relative error $\epsilon$ at rank $n_0$. We define the error amplification factor $A$ by
$$A=\min \left\lbrace \tilde{A}\in\mathbb{R}^*_+ \hbox{ such that } \left| \frac{\sum\limits_{i=1}^p f_i(n)}{\sum\limits_{i=1}^p F_i(n)} -1 \right| \leq \tilde{A}\epsilon \quad \forall n\geq n_0 \right\rbrace .$$
\end{defi}

\begin{lem}
We consider $p$ functions $f_1,\dots,f_p:\mathbb{N} \longrightarrow \mathbb{R}$ and approximations $F_1,\dots,F_p\in\mathbb{R}(n)[H(n)]$ with relative error $\epsilon<1$ at rank $n_0$ and $A$ their amplification factor. Then
$$A \leq \max\limits_{n\geq n_0} \frac{\sum\limits_{i=1}^p \left| F_i(n) \right| }{\left| \sum\limits_{i=1}^p F_i(n) \right| }.$$
\end{lem}

\begin{proof}
The lemma is equivalent to prove that
$$\left| \frac{\sum\limits_{i=1}^p f_i(n)}{\sum\limits_{i=1}^p F_i(n)} -1 \right| \leq \epsilon \max\limits_{n\geq n_0} \frac{\sum\limits_{i=1}^p \left| F_i(n) \right| }{\left| \sum\limits_{i=1}^p F_i(n) \right| }$$
So one just needs to maximize the left hand side. We already know that $|f_i(n)/F_i(n) -1| \leq \epsilon$. So depending on the sign of $f_i(n)$ we replace $f_i(n)$ by $(1-\epsilon)F_i(n)$ or $(1+\epsilon)F_i(n)$. We then expand
\begin{align*}
\left| \frac{\sum\limits_{i=1}^p f_i(n)}{\sum\limits_{i=1}^p F_i(n)} -1 \right| \leq \left|\frac{\epsilon \sum\limits_{i=1}^p \hbox{sign}(f_i(n)) F_i(n)}{\sum\limits_{i=1}^p F_i(n)} \right|
\leq \left|\frac{\epsilon \sum\limits_{i=1}^p \left|F_i(n)\right|}{\sum\limits_{i=1}^p F_i(n)} \right|\\
\leq \epsilon \max\limits_{n\geq n_0} \frac{\sum\limits_{i=1}^p \left| F_i(n) \right| }{\left| \sum\limits_{i=1}^p F_i(n) \right| }
\end{align*}
using the fact that $f_i(n)$ and $F_i(n)$ have always the same sign for $n\geq n_0$ (because $\epsilon<1$).
\end{proof}

In practice, we first check that the sign of the functions $F_i(n)$ and their sum does not change for $n\geq n_0$ and then we prove a majoration of the resulting expression in $\mathbb{R}(n,H(n))$. So all comes down to prove that some polynomial in $\mathbb{R}[n,H(n)]$ does not vanish for $n\geq n_0$. This can be done by first making an encadrement of the function $H(n)$ and then prove that the corresponding bivariate polynomial does not vanish on a particular algebraic subset. Such a problem can be algorithmically decided.

\begin{thm}\label{thm:maj}
Consider the recurrence equation
\begin{equation}\label{eqrecurrence}
u(n+1)=A(n)u(n) \quad \forall n \in\mathbb{N}, \;A(n) \in M_p(\mathbb{C})
\end{equation}
Consider $\lVert \cdot \rVert$ a matricial norm and $R(n)$ the resolvant matrix of equation~\eqref{eqrecurrence}. Assume that
$$M(\infty)=\sum\limits_{j=0}^\infty \lVert A(j)-I_p \rVert\; < 1 .$$
Then
$$\lVert R(n) -I_p \rVert \leq \frac{M(\infty)}{1-M(\infty)} \quad \forall n\in\mathbb{N}.$$
\end{thm}

\begin{proof}
We write
$$R(n)=\prod\limits_{i=0}^{n-1} A(i)= \prod\limits_{i=0}^{n-1} ((A(i)-I_p) +I_p).$$
Let us pose
$$M(n)=\sum\limits_{j=0}^{n-1} \lVert A(j)-I_p \rVert .$$
We want to prove a majoration of the type
\begin{equation}\label{majoration}
\lVert R(n) -I_p \rVert \leq C M(n)
\end{equation}
with a suitable constant $C>0$. For $n=1$, this is true with $C=1$. Let us prove equation~\eqref{majoration} by recurrence:
$$R(j)=\prod\limits_{i=0}^{j-1} ((A(i)-I_p) +I_p)= (A(j-1)-I_p)\prod\limits_{i=0}^{j-2} A(i) +\prod\limits_{i=0}^{j-2} A(i), $$
$$R(j)-R(j-1)= (A(j-1)-I_p)\prod\limits_{i=0}^{j-2} A(i)= (A(j-1)-I_p) R(j-1).$$
Then we sum these equations for $1\leq j\leq n$ which produces
\begin{align*}
\lVert R(n)-I_p \rVert= \left\lVert \sum\limits_{j=0}^{n-1} (A(j)-I_p)(R(j)-I_p) +(A(j)-I_p) \right\rVert \\
 \leq \sum\limits_{j=0}^{n-1} \lVert A(j)-I_p \rVert\, \lVert R(j)-I_p\rVert +\lVert A(j)-I_p\rVert\\
\leq M(n) + \sum\limits_{j=0}^{n-1} \lVert A(j)-I_p \rVert\, C M(j)
=M(n)+ C M(n)^2\\
\leq (1+C M(\infty)) M(n)
\end{align*}
using the fact that $M(n)$ is a growing sequence. So the recurrence property is proved if $C\leq 1+C M(\infty)$ which is equivalent to $C\geq (1-M(\infty))^{-1}\geq 1$. So this proves that
$$\lVert R(n) -I_p \rVert \leq \frac{M(n)}{1-M(\infty)} \leq \frac{M(\infty)}{1-M(\infty)}$$
which proves the theorem.
\end{proof}

The main application of this theorem is to compute a sequence with
controlled error. Let us take an operator $L\in\mathbb{R}\langle
n,S_n\rangle$ regular at infinity. We can then compute an
asymptotic expansion of the resolvant matrix of $L$, and an error
matrix which will satisfy an equation like~\eqref{eqrecurrence}. Then
for an $n_0\in\mathbb{N}$, we can apply Theorem~\ref{thm:maj} for the
shifted sequence $u(n+n_0)$, and the majoration $M(\infty)$ will
become very small for $n_0$ big enough, giving us that the error is
always lower than some explicit bound. This has very important
consequences for the application of the higher variational method. In
particular, it becomes possible to rigorously prove that a sequence of
potentials with the unbounded eigenvalue property does not satisfy
integrability criteria for $\lambda$ large enough, and thus coming
back to a bounded eigenvalue problem.

\section{Application at Order $2$}

We now apply the second order criterion to our example. We begin with
the case~$E_4$. Before we state the corresponding theorem, we need a
preparatory lemma concerning the solutions of a certain Diophantine
equation.

\begin{lem} \label{lem:diophantine1}
The set of solutions $(k_1,k_2)\in\mathbb{N}^2$ of the Diophantine equation
\[
  R(k_1,k_2)=k_2^2k_1^2+k_2k_1^2-75k_1^2-75k_1+k_2k_1-27k_2+k_2^2k_1-27k_2^2=0
\]
is given by $\{(0,0),\,(6,14)\}$.
\end{lem}
\begin{proof}
We begin by proving that for $k_2\geq 50$, the condition $R=0$ implies $4<k_1<5$,
and similarly, for $k_1\geq 50$, we have $8<k_2<9$. These statements can be
written as logical expressions involving polynomial inequalities
\begin{align}
& \forall k_1\forall k_2: (k_1\geq0\land k_2\geq50\land R(k_1,k_2)=0)\implies 4<k_1<5,\label{eq.cad1}\\
& \forall k_1\forall k_2: (k_1\geq50\land k_2\geq0\land R(k_1,k_2)=0)\implies 8<k_2<9.\label{eq.cad2}
\end{align}
Such formulas can be proven routinely with quantifier elimination
techniques like cylindrical algebraic
decomposition~\cite{Collins75}. Indeed, applying the Mathematica command
\textbf{CylindricalDecomposition} to the above formulae reveals that
they are true. Therefore, there are no integer solutions for
$k_1\geq50$ or $k_2\geq50$ and an exhaustive search delivers exactly
the solutions claimed above (Figure \ref{figg}.

However, if we want to prove \eqref{eq.cad1} and \eqref{eq.cad2} ``by
hand'' (let's consider the first one for the moment), we have to look
at the largest real root of the polynomial
\[
  \res_{k_1}\!\left(\!R(k_1,k_2),\,\frac{\partial R(k_1,k_2)}{\partial k_1}\right)R(4,k_2)\,R(5,k_2).
\]
We find that this root is smaller than 50 (using real root isolation)
and that the limit
\[
  \lim\limits_{k_2 \rightarrow\infty} \kappa(k_2)=-\frac{1}{2}+\frac{1}{2}\sqrt{109}
\]
is between $4$ and $5$, where $\kappa(k_2)$ denotes the positive
solution of $R(k_1,k_2)=0$ regarded as an equation in~$k_1$. The
implication~\eqref{eq.cad1} follows, and \eqref{eq.cad2} can be proven
analogously.
\end{proof}

\begin{thm} \label{thm:order21} 
We consider the potential~$E_4$ given in Theorem~\ref{thm.families}.
If the variational equation near all Darboux points is integrable at
order~$2$, then the corresponding eigenvalues are integers of the form
$\lambda=(2l-1)(l+1),\;\;l\in\mathbb{N}$.
\end{thm}
\begin{proof}
We use the notation $U=rE_4$ from Theorem~\ref{thm.families}.
The condition $U'(\theta)=0$ yields the two Darboux points
\begin{equation}\label{eq8}
\begin{split}
c_1: e^{i\theta} & =1,\\
c_2: e^{i\theta} & =\frac{s+6\lambda_1}{s-6\lambda_2}.
\end{split}
\end{equation}
There are singular cases of the second equation, namely for
$s+6\lambda_1=0$ or $s-6\lambda_2=0$. After solving and replacing, we
find that these cases correspond exactly to $k_1=0$ and $k_1=3$, which
were excluded from~$E_4$.

We now compute the third derivative of~$V$, evaluated at the two
Darboux points $c_1$ and~$c_2$ given by expression~\eqref{eq8}:
\begin{align*}
  \frac{\partial^3 V}{\partial q_2^3}(c_1) & =
    \frac{i\lambda_1(s+15\lambda_1+9\lambda_2)}{\lambda_1+\lambda_2},\\ 
  \frac{\partial^3 V}{\partial q_2^3}(c_2) & = 
    -\frac{i\lambda_2(s-15\lambda_2-9\lambda_1)}{3(\lambda_1+\lambda_2)}.
\end{align*}
In the case $(k_1,k_2)$ both odd, both derivatives should vanish. We solve the system and we find $4i(k_2+1)k_2=0$. This is impossible for odd values. In the case $k_1$ odd $k_2$ even, the first one should vanish, and in the case $k_1$ even $k_2$ odd the second one should vanish. We get the equations
\begin{equation}\begin{split}\label{eq10}
\frac{k_1^2(k_1+1)^2(k_2^2k_1^2+k_2^2k_1-27k_2^2-27k_2-75k_1+k_2k_1-75k_1^2+k_2k_1^2)}{12(k_2^2+k_2+k_1+k_2^2)}\\
\frac{k_2^2(k_2+1)^2(k_1^2k_2^2+k_1^2k_2-27k_1^2-27k_1-75k_2+k_1k_2-75k_2^2+k_1k_2^2)}{12(k_1^2+k_1+k_2+k_2^2)}\\
\end{split}\end{equation}

These two conditions are symmetric. The first terms can never vanish because we have $k_1$ odd for the first one and $k_2$ odd for the second one. To conclude, we need to look at the last term, which corresponds to a Diophantine equation, and to prove that this equation does not have a solution with $k_1$ odd and $k_2$ even.

With Lemma~\ref{lem:diophantine1}, we have no solutions from the second term where $k_1$ and $k_2$ have different parity. We conclude that all the possibilities left are for $k_1,k_2$ even.
\end{proof}

It is well known that Diophantine equations in general cannot be
solved (Matiyasevich's theorem). This means that
Lemma~\ref{lem:diophantine1} is a lucky case, although not trivial to
prove. We therefore should remark that the study of this equation is
not absolutely mandatory. We could simply skip it, \textbf{assume}
that it is satisfied and continue further to the third-order
condition. This condition would add two additional equations in $k_1$
and $k_2$ and thus would allow to solve the problem in all generality.

Here we are in a special case. A Diophantine equation $R(k_1,k_2)=0$
can be solved only using real algebraic geometry in one of the
following cases:
\begin{enumerate}
\item The set $R^{-1}(0)\cap\mathbb{R^+}^2$ is compact. In this case
  we only have a finite number of points to test.
\item The set $R^{-1}(0)\cap\mathbb{R^+}^2$ is not compact but all
  infinite branches are asymptotes and the corresponding asymptotic
  straight lines have a rational slope. In this case, either $R$ is
  homogeneous and has an infinite number of solutions, or the integer
  solutions can be bounded: when approaching infinity, the infinite
  branch of $R^{-1}(0)$ comes closer to the asymptotic line without
  touching it; for rational slope, there is then a nonzero infimum for
  the distance between the asymptotic straight line and integer
  points).
\end{enumerate}
The first case can be considered to be part of the second one with no
asymptotes at all. In Lemma~\ref{lem:diophantine1}, we encounter the
second case.

\begin{figure}\label{figg}
\begin{center}
\includegraphics[width=7cm]{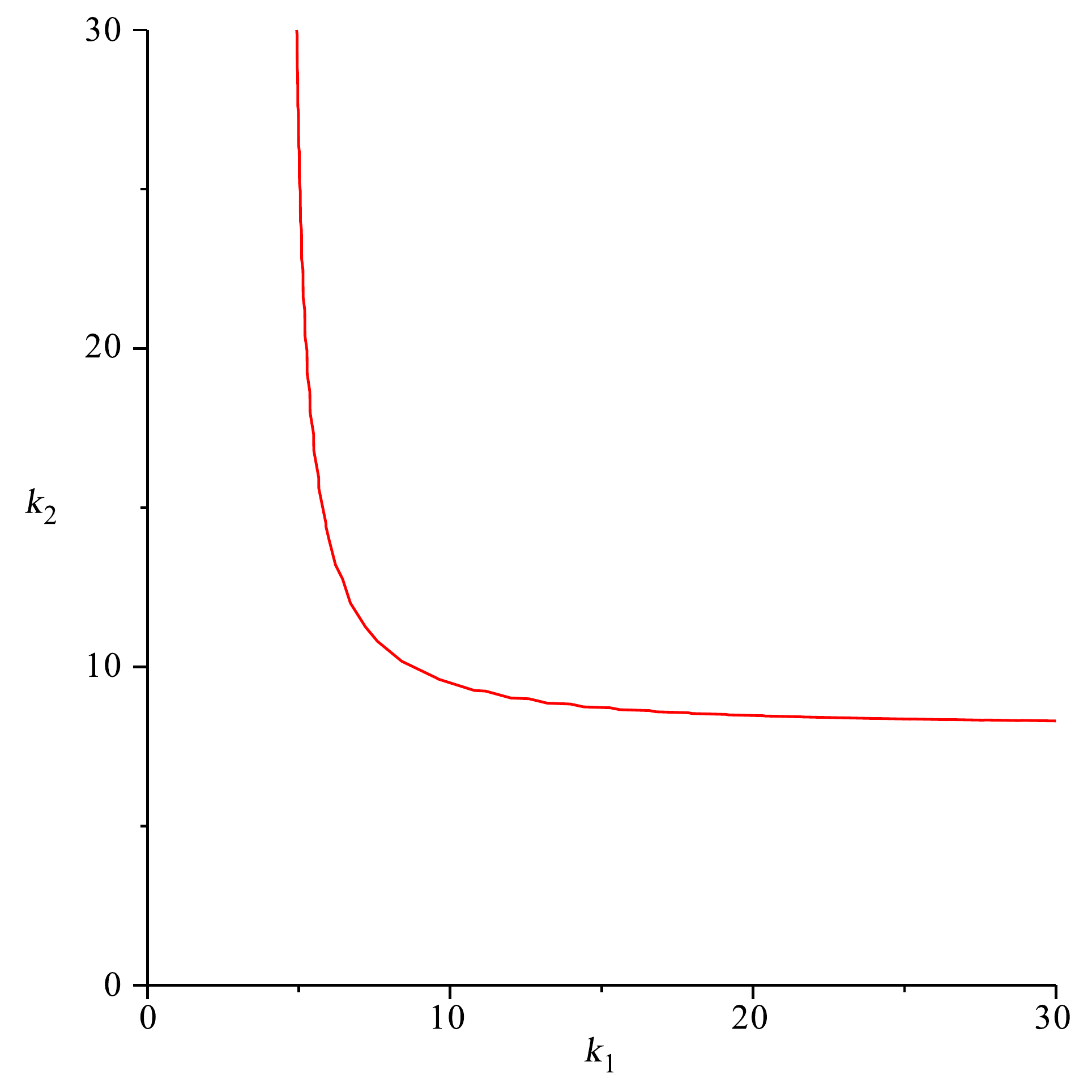}
\end{center}
\caption{Graph of $R^{-1}(0)$. The graph $R^{-1}(0)\cap\mathbb{R^+}^2$ is not compact but the infinite branches are asymptotes with rational or vertical slopes; here the asymptotes are $k_1+\frac{1}{2}-\frac{1}{2}\sqrt{109}=0$ and $k_2+\frac{1}{2}-\frac{1}{2}\sqrt{301}=0$.}
\end{figure}

\begin{rem}
The potential corresponding to $k_1,k_2=(6,14)$ is the following (with the good choice of valuation for the square root):
$$V(r,\theta)=\frac{1}{r}\left(-20+\frac{105}{2}e^{i\theta}-42e^{2i\theta}+\frac{21}{2}e^{3i\theta} \right).$$
This potential has two Darboux points, it is integrable at order $2$ near these two Darboux points and we have also that the third derivative near one of the Darboux points is zero (which is not needed for integrability at order $2$ but gives interesting properties in practice at order $3$).
\end{rem}

\begin{thm} \label{thm:order22}
Among the potentials in the families $E_1,E_2,E_3$, if a potential $V$ is meromorphically integrable, then it is of the form (after multiplying by some constant factor):
\begin{align*}
V & = \frac{1}{r}\left( -\frac{1}{3}k(2k+1)e^{3i\theta}+\frac{1}{2}k(2k+1)e^{2i\theta}-\frac{1}{6}(2k^2+k-6)\right),\\
V & = \frac{1}{r}\left(-\frac{1}{2}k(2k+1)e^{2i\theta}+k(2k+1)e^{i\theta}-\frac{1}{2}(2k^2+k-2)\right),
\end{align*}
for $k\in\mathbb{N}$.
\end{thm}

\begin{proof}
The potentials $E_2$ and $E_3$ possess only one Darboux point. The corresponding potentials are
\begin{align*}
E_2:\; V & = r^{-1}\left(-\frac{1}{6}k(k+1)e^{3i\theta}+\frac{1}{4}k(k+1)e^{2i\theta}-\frac{1}{12}k^2-\frac{1}{12}k+1\right),\\
E_3:\; V & = r^{-1}\left(-\frac{1}{4}k(k+1)e^{2i\theta}+\frac{1}{2}k(k+1)e^{i\theta}-\frac{1}{4}k^2-\frac{1}{4}k+1\right).
\end{align*}
We know that if $k$ is odd, we have an additional integrability condition at order~$2$. We find that
$$\frac{\partial^3V}{\partial q_2^3}(c)=\frac{5}{2}ik(k+1) \hbox{ for } E_2,\qquad \frac{\partial^3V}{\partial q_2^3}(c)=\frac{3}{2}ik(k+1) \hbox{ for } E_3.$$
These terms should vanish. This is never fulfilled for odd~$k$. The sequence of potentials given by Theorem~\ref{thm:order22} corresponds exactly to the cases of even~$k$ (for which there is no condition for integrability at order~$2$). At last, we have the potential~$E_1$. The corresponding eigenvalue is always $-1$, so it is always integrable at order $2$. At order $3$, we know that the integrability condition is $U^{(3)}(0)=0$. We get
$$U^{(3)}(0)=-2ib$$
So the only possibility is $b=0$ and this corresponds to the potential $V=r^{-1}$. This potential is integrable and already belongs to the family described by Theorem~\ref{thm:order22}.

\end{proof}

\section{Application at Order $3$}

We will now prove Theorem~\ref{thm:main3}, building an algorithm to prove it. 
\begin{proof}
The scheme of the proof is the following
\begin{itemize}
\item First we prove that the recurrences for $f_1,f_2,f_3$ are regular at infinity.
\item We then produce a series expansion $\tilde{R}_i(n)$ at infinity at an order high enough of the resolvant matrix $R_i(n)$ associated to these recurrences.
\item We then write $R_i(n)=\tilde{R}_i(n) \tilde{R}_i(n_0)^{-1} R_i(n_0)  E_i(n)$ for a large enough $n_0\in\mathbb{N}$ and build a recurrence of the form~\eqref{eqrecurrence} whose resolvant matrix is $E_i(n)$ (after change of basis), which will be denoted by $E_i(n+1)=A_i(n) E_i(n)$. We have moreover that $E_i(n_0)$ is the identity matrix.
\item As $\tilde{R}_i(n)$ is a good approximation of $R_i(n)$ when $n\longrightarrow \infty$, the matrix $A_i(n)$ will tend to the identity matrix when $n\longrightarrow \infty$. Using Theorem~\ref{thm:maj} with a shift in the indices, we will have that
$$\lVert E_i(n) -I \rVert \leq \frac{\sum\limits_{j=n_0}^\infty \lVert A_i(j)-I \rVert}{1-\sum\limits_{j=n_0}^\infty \lVert A_i(j)-I \rVert } \quad \forall n\geq n_0$$
\item If we have chosen an expansion order and $n_0$ large enough, this sum will be finite and small, and thus will give us an approximation of $R_i(n)$ by $\tilde{R}_i(n)$ with relative error control. The expressions in Theorem~\ref{thm:main3} follow.
\end{itemize}

For $f_3(2n)$, we find the following asymptotic expansion (a high order makes up the computation easier for error control)
\begin{align*}
c_1 \left(\frac{1}{n^4}-\frac{1}{n^5}+\frac{25}{32n^6}-\frac{35}{64n^7}+\frac{183}{512n^8}\right)+\\
c_2\left(\left(\frac{3}{16n^4}-\frac{3}{16n^5}+\frac{75}{512n^6}-\frac{105}{1024n^7}+\frac{549}{8192n^8}\right)H(n)+\right.\\
\left. \frac{1}{n^2}-\frac{1}{2n^3}+\frac{19951}{46848n^4}-\frac{7507}{46848n^5}+\frac{96541}{1499136n^6}-\frac{58151}{2998272n^7} \right)
\end{align*}
This proves by the way that the recurrence for $f_3(2n)$ is regular. We do the same for $f_1(2n)$ and $f_2(2n)$ and we find that they are regular too. We then find a majoration of the norm of the error matrix $A_3(n)$
\begin{align*}
\lVert A_3(n) \rVert_{\infty} \leq 
\frac{9975}{256n^6}+\frac{29925}{4096}\frac{H(n)}{n^6}+\frac{9975}{256n^8}+\frac{29925}{4096}\frac{H(n)}{n^8}
\end{align*}
We choose now $n_0=100$. We majorate the sum of this majoration beginning at $n=100$. We find a majoration of this sum by 
$$\sum\limits_{n=100}^\infty \lVert A_3(n) \rVert_{\infty} \leq 4.84522\times 10^{-9}$$
$$\lVert E_3(n) \rVert \leq \frac{4.84522\times 10^{-9}}{1-4.84522\times 10^{-9}} \qquad \forall n\geq n_0$$
(an explicit rational number). We then compute the recurrence up to $n=100$, and then produce an encadrement (with error less than $10^{-10}$) of the result with rational numbers. Although it is not mandatory in theory, in practice recurrences tend to produce very large rational numbers, whose size grows linearly with $n$, and thus are impractical to manipulate. This gives us the coefficients $c_1,c_2$ with a good error control:
$$c_1=-\frac{883919839}{274877906944},\qquad c_2=-\frac{1740684681}{8589934592}.$$
We then compute the error amplification of the sum, and find that it is less than $33/32$. As the resulting expression is too complicated to manipulate for applications, we only keep the terms up to order~$3$ and prove that this new approximation has a relative error less than $10^{-5}$. The expressions for $f_1$ and $f_2$ are found with a similar way, with the exception that at the end, to produce a sufficiently simple and accurate formula, it is not sufficient to keep the terms up to order~$2$ (after there is a $H(n)$ that we want to avoid), so we need to add a term of order~$3$ (without $H(n)$) with a well chosen coefficient such that the error stays below $10^{-5}$ (else the result is only accurate to $10^{-3}$).
\end{proof}

\begin{thm} \label{thm:order30}
The third order integrability conditions for the families
\begin{align*}
V & = \frac{1}{r}\left( -\frac{1}{3}k(2k+1)e^{3i\theta}+\frac{1}{2}k(2k+1)e^{2i\theta}-\frac{1}{6}(2k^2+k-6)\right)\\
V & = \frac{1}{r}\left(-\frac{1}{2}k(2k+1)e^{2i\theta}+k(2k+1)e^{i\theta}-\frac{1}{2}(2k^2+k-2)\right)
\end{align*}
where $k\in\mathbb{N}^*$, are 
\begin{align*}
9(k+1)^2(2k-1)^2f_1(2k) & = 25k^2(2k+1)^2f_2(2k) + (66k^2+33k-9)f_3(2k),\\
9(k+1)^2(2k-1)^2f_1(2k) & = 9k^2(2k+1)^2f_2(2k) + (42k^2+21k-9)f_3(2k),
\end{align*}
respectively. They are never satisfied.
\end{thm}

\begin{proof}
We replace $f_1(2k),f_2(2k),f_3(2k)$ by their approximations, and then
compute the error amplification. It is less than $33/32$, and the
resulting expression does not vanish for $k\geq 100$. For $k<100$, we
make exhaustive testing and we do not find any solutions. For the
second equation, we do not find any solution either.
\end{proof}

\begin{thm} \label{thm:order31}
We consider the family of potentials~$E_4$
\begin{align*}
E_4:\quad V=r^{-1}\left(\frac{(s-6\lambda_2)\lambda_2}{18(\lambda_1+\lambda_2)}e^{3i\theta}-\frac{(3\lambda_1+s-3\lambda_2)\lambda_2}{6(\lambda_1+\lambda_2)}e^{2i\theta}+\right.\\
\left.\frac{(6\lambda_1+s)\lambda_2}{6(\lambda_1+\lambda_2)}e^{i\theta}+\frac{-9\lambda_1\lambda_2-\lambda_2s+18\lambda_1+18\lambda_2-3\lambda_2^2}{18(\lambda_1+\lambda_2)}\right)
\end{align*}
with
$$s^2=6\lambda_1^2\lambda_2+6\lambda_1\lambda_2^2-36\lambda_1\lambda_2\qquad \lambda_1=\frac{1}{2}(k_1-1)(k_1+2)+1$$
$$\lambda_2=\frac{1}{2}(k_2-1)(k_2+2)+1\quad k_1,k_2\in\mathbb{N}^*\;\;k_1\neq 3$$
The third order integrability condition for $E_4$ is of the form
$$Q_{k_1,k_2}(f_1(k_1),f_2(k_1),f_3(k_1))=0$$
$$Q_{k_2,k_1}(f_1(k_2),f_2(k_2),f_3(k_2))=0$$
where $Q$ is a quadratic form depending polynomially on $k_1$ and $k_2$.
\end{thm}

\begin{proof}
We use Theorem~\ref{thm:main1} and compute the derivatives of the
potentials in the family~$E_4$. These derivatives depend rationally
on $k_1$, $k_2$, and~$s$. As there are two Darboux points, we get two conditions
$(C_1),(C_2)$ linearly dependent on $f_1(k_1),f_2(k_1),f_3(k_1)$ or
$f_1(k_2),f_2(k_2),f_3(k_2)$ respectively for each Darboux point. To
remove the quadratic extension $s$, we make the product $(C_1)\times
\hbox{subs}(s=-s,(C_1))$ and $(C_2)\times
\hbox{subs}(s=-s,(C_2))$. The fact that in the potentials of $E_4$,
the two parameters $\lambda_1$ and $\lambda_2$ play a symmetric r\^{o}le produces
the two conditions $Q_{k_1,k_2}=0$ and $Q_{k_2,k_1}=0$.
\end{proof}

\begin{rem}
The conditions $Q_{k_1,k_2},Q_{k_2,k_1}$ are not equivalent to the conditions $(C_1),(C_2)$. We can solve $(C_1)$ in the quadratic extension and get for example that $s$ should be rational because $f_1,f_2,f_3$ are always rational (this can be proven even without the P-finite recurrences since they correspond to a particular term in the series expansion of rational expressions in $t,P_n(t),Q_n(t)$). We get that 
\begin{equation}\begin{split}
\sqrt{3 k_1k_2(k_2+1)(k_1+1)(k_1+k_1^2+k_2+k_2^2-12)}\;\;\in\mathbb{N}
\end{split}\end{equation}
if some generic condition depending on the $f_i(k_2),f_i(k_1)$ is satisfied. It corresponds to a Diophantine equation but it does not possess the nice properties we used to solve Lemma~\ref{lem:diophantine1}. We know moreover that $(k_1,k_2)$ should be even. A direct search produces the picture given in Figure~\ref{fig.dio2}.
\end{rem}

\begin{figure}
\begin{center}
\includegraphics[width=8cm]{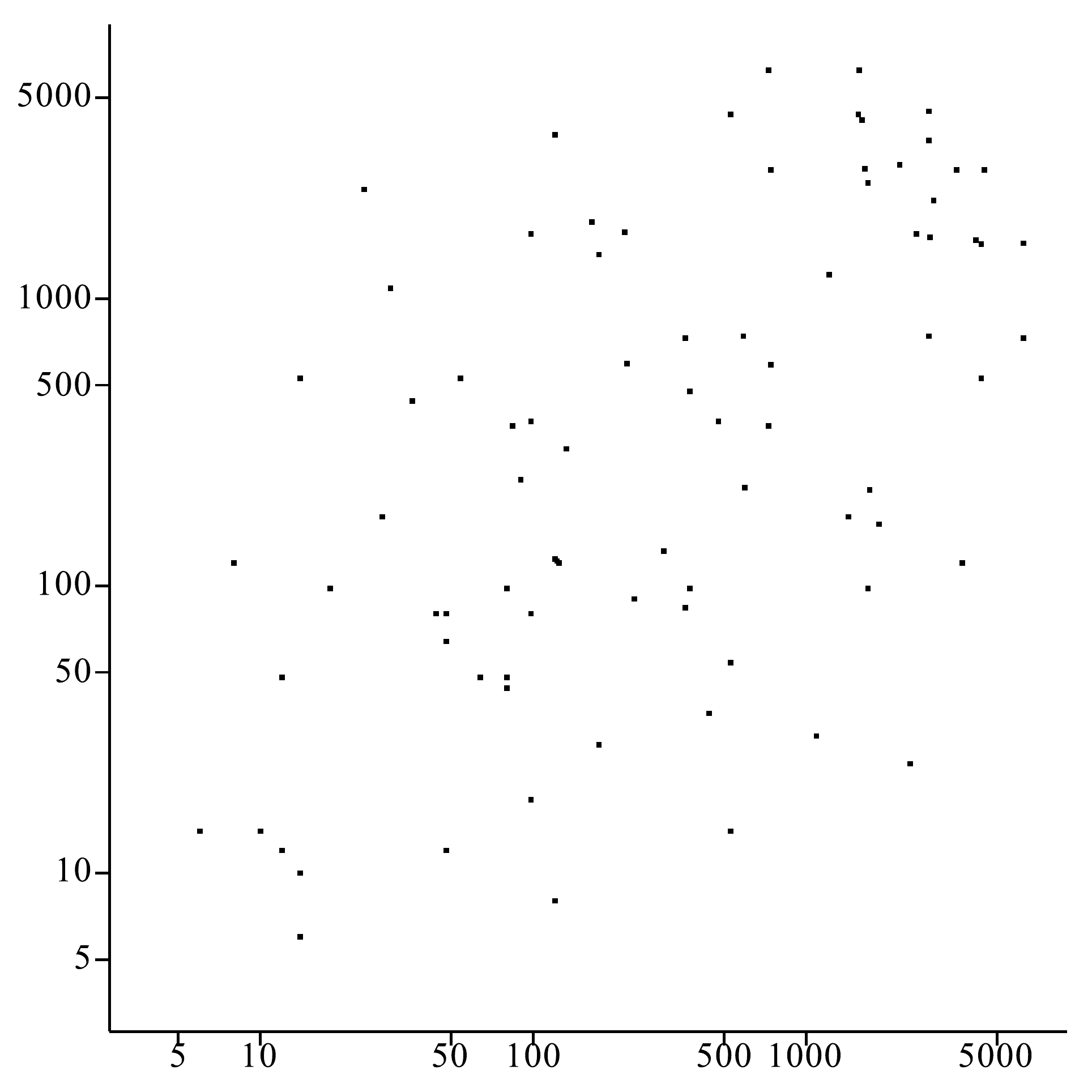}
\end{center}
\caption{Each dot corresponds to a possible even $(k_1,k_2)$. The solutions seems to be unbounded, and the set is probably Zarisky dense. Thus, no practical algebraic information can be extracted from this constraint. There are infinitely many solutions (because the diagonal part reduces to a Pell equation) with simultaneously arbitrarily high $k_1,k_2$. So here Theorem~\ref{thm:main1} is useless without Theorem~\ref{thm:main3}.}
\label{fig.dio2}
\end{figure}

\begin{thm} \label{thm:order32}
The third order integrability condition for $E_4$ is never satisfied except for $(k_1,k_2)=(2,2)$.
\end{thm}

\begin{proof}
Recall that the parameters $(k_1,k_2)$ need to be both even for a potential $E_4$ to be integrable at order $2$ near all Darboux points. We begin by solving $Q_{k_2,k_1}(f_1(k_2),f_2(k_2),f_3(k_2))=0$ in $k_1$. This is a polynomial of degree $4$ in $k_1$ and as a polynomial, its Galois group is $D_4$. This allows us to write the 
solution in a relatively simple form
\begin{equation}\begin{split}\label{eq11}
k_1=-\frac{1}{2}+\sqrt{F_1(k_2)+wF_2(k_2)}\quad\hbox{ with }\quad w^2=\\
9(k_2+2)^2(k_2-1)^2f_1(k_2)f_2(k_2)-6(k_2+3)(k_2-2)f_2(k_2)f_3(k_2)+36f_3(k_2)^2
\end{split}\end{equation}
where $F_1,F_2\in\mathbb{Q}(f_1,f_2,f_3,k_2)$. Moreover, $k_1,k_2$ are even integers. Let us prove that in fact, for even $k_2\geq 200$, the expression
$$-\frac{1}{2}+\sqrt{F_1(k_2)+wF_2(k_2)}$$
is always complex for all possible valuations of the square roots. To have real values, we need that $F_1(k_2)+wF_2(k_2)$ be positive for at least one valuation of the square root. Let us begin by proving that $w$ never vanishes. The function $w^2$ is a polynomial in $\mathbb{Q}[f_1,f_2,f_3,k_2]$. Thanks to Theorem~\ref{thm:main3}, we can express $f_1,f_2,f_3$ in $k_2$ with controlled relative error. We check that the amplication of the error is small after summation of all terms (here it is less than $1+10^{-3}$) and that the approximated expressions never vanish. Now we need to prove that
$$F_1(k_2)+wF_2(k_2)< 0 \hbox{ and } F_1(k_2)-wF_2(k_2)< 0.$$
We first prove that $F_2(k_2)$ and $F_1(k_2)$ (which are in $\mathbb{Q}[f_1,f_2,f_3,k_2]$ of degree $3,4$ in $f_i$ respectively) are always negative. Then we just have to prove that
\begin{align*}
\frac{F_1(k_2)} {wF_2(k_2)} > 1 
\;\Longleftrightarrow \;\frac{F_1(k_2)^2} {w^2F_2(k_2)^2}> 1 
\;\Longleftrightarrow\; F_1(k_2)^2-w^2F_2(k_2)^2> 0.
\end{align*}
The last expression is in $\mathbb{Q}[f_1,f_2,f_3,k_2]$ (of degree $8$ in $f_i$), so we can prove this statement. Again we compute the error amplification of the sum and it stays below $1+10^{-3}$, and the error is then still less than $10^{-4}$. Eventually, we prove that this approximated expression never vanishes and is always positive.
For the remaining cases, we use exhaustive testing and we find only one solution $(k_1,k_2)=(2,2)$.
\end{proof}

The case $(k_1,k_2)=(2,2)$ corresponds to the second case of Theorem~\ref{thm:main2}. It is really integrable with a quadratic in momenta additional first integral which is given in \cite{6} page 107 case (8).

\section{Remaining Cases and Conclusion}

The remaining cases are the ones which do not posess a non-degenerate Darboux point.
\begin{thm}\label{thm:exceptions} Consider the set of potentials $V$ given by \eqref{mat_eq} and suppose that $V$ does not possess a non-degenerate Darboux point $c$. If $V$ is meromorphically integrable, then $V$ belongs to one of the families
\begin{align*}
V & =\frac{1}{r}\left( a+be^{i\theta}\right), & V =\frac{1}{r}\left( a+be^{2i\theta}\right),\\
V & =\frac{1}{r}\left( a+be^{3i\theta}\right), & V =\frac{1}{r}\left( a+be^{i\theta}\right)^3,
\end{align*}
with $a\in\mathbb{C}$, $b\in\mathbb{C}^*$.
\end{thm}

\begin{proof}
First let us suppose that $V$ does not possess any Darboux point $c$. This means that the function
$$U(\theta)=a+be^{i\theta}+ce^{2i\theta}+de^{3i\theta}$$
does not possess any critical point. The only possibility is that $U(\theta)=F(e^{i\theta})$ with $F(z)=a+bz^n$, $b\neq 0$. This corresponds to the three first cases of Theorem~\ref{thm:exceptions}. Now suppose there exists one Darboux point $c$ but degenerate. After rotation, we can suppose that the Darboux point corresponds to $\theta=0$. We have moreover the integrability constraint that $U''(0)=0$. This gives the potential
$$V=\frac{a}{r}\left(e^{i\theta}-1\right)^3.$$
After rotation, this corresponds to the fourth case of Theorem~\ref{thm:exceptions}.
\end{proof}

Remark that in contrary to previous sections, we use the hypothesis that the first integrals are meromorphic on $\mathbb{C}^2\times\mathcal{S}$, including $r=0$. If we consider first integrals only meromorphic on $\mathbb{C}^2\times\mathcal{S}^*$, the integrability constraint $U''(0)=0$ does not hold anymore and two additional cases appear in Theorem \ref{thm:main2}
$$V =\frac{1}{r}\left( a+be^{i\theta}\right)^2, \qquad \qquad V =\frac{1}{r}\left(a+be^{i\theta}\right)^2\left(a-2be^{i\theta}\right)$$

The family $V=\frac{1}{r}\left( a+be^{i\theta}\right)$ is integrable as given in \cite{6}, but with an additional first integral only meromorphic on $\mathbb{C}^2\times\mathcal{S}^*$. For the other ones, the integrability status is still unknown. Let us remark now on the open cases. After rotation and dilatation, these cases correspond in fact to a finite number of potentials which are the following (the last two cases being open only for meromorphic first integrals on $\mathbb{C}^2\times\mathcal{S}^*$):
{\setlength{\arraycolsep}{0.15em}
\begin{equation}\begin{array}{rclrclrcl}
V & = & r^{-1}e^{2i\theta},\qquad  &
V & = & r^{-1}\left(e^{2i\theta}-1\right),\qquad  &
V & = & r^{-1}\left(e^{3i\theta}-1\right),\\[0.5em]
V & = & r^{-1}\left(e^{i\theta}-1\right)^3,\qquad &
V & = & r^{-1}\left(e^{i\theta}-1\right)^2,\qquad &
V & = & r^{-1}\left(e^{i\theta}-1\right)^2\left(2e^{i\theta}+1\right).
\end{array}\end{equation}

We cannot study these potentials because we do not have a particular solution to study, or a sufficiently non-degenerate one (studying degenerate Darboux points with higher variational method is in fact useless and does not give any additional integrability condition). This is of course the main weakness of the Morales-Ramis theory. This is not due to the difficulty of applying the Morales-Ramis theory as we treat it in this article, but much more a fundamental limitation that seems hard to overcome. One approach could consist in looking for special algebraic orbits of these systems using a direct search (following Hietarinta \cite{6}). This is not successful for all these potentials.

To conclude, let us remark that our holonomic approach to higher
variational methods is very general, and in no way limited to this
example. This could work at least for all problems about integrability
of homogeneous potentials, as it allows to compute various higher
integrability conditions of any fixed order. This is linked to the
fact that the first order variational equation of a natural
Hamiltonian system often corresponds to a spectral problem of a second
order differential operator, which generates P-finite sequences of
functions, which in turn appear in the study of higher variational
equations. We could also wonder if these arbitrary high eigenvalues
are really possible, and if this work is only conceptual and in
practice useless. Indeed, very high eigenvalues should correspond to
very high degree first integrals, and counting the number of
conditions and number of free parameters for the existence of such
high degree first integrals strongly suggests they do not exist. But
this intuition is wrong, as Andrzej J. Maciejewski, Maria Przybylska
found quite recently such an example in dimension~$3$. This is
probably linked to the fact that most of integrable cases come from
ultra-degenerate cases, as in our analysis: the generic case $E_4$
contains only one possibility, and when we look at the third order
integrability condition, it seems really to be a miracle that this
condition could ever be satisfied. On the contrary, the cases without
Darboux points contain lots of integrable potentials.

\end{document}